\documentclass{amsart}
\usepackage[utf8]{inputenc}
\usepackage{amsmath,amssymb,amscd,mathrsfs}
\usepackage{stmaryrd}
\usepackage{graphicx}
\usepackage{relsize}
\usepackage{cancel}
\usepackage{geometry}
\usepackage{mathtools}
\usepackage{tikz,tkz-tab,tikz-cd}
\usepackage{xcolor,colordvi,multicol,fancybox,tcolorbox}
\usepackage{dsfont}
\usepackage{enumitem}
\usepackage{boxedminipage}
\usepackage{scalerel,stackengine}
\usepackage{bookmark}
\usepackage[type={CC},modifier={by-nc-sa},version={4.0},]{doclicense}
\usepackage[all,cmtip]{xy}
\usepackage[normalem]{ulem}
\usepackage[backend = bibtex, style = alphabetic, maxbibnames = 5]{biblatex} 
\tcbuselibrary{skins,breakable}
\usepackage{graphicx}
\usepackage{titlepic}
\usepackage{yhmath}
\usepackage{thmtools}

\def\dar[#1]{\ar@<2pt>[#1]\ar@<-2pt>[#1]}
\entrymodifiers={!!<0pt,0.7ex>+}

\title{Geometric obstructions to fully ellipticity for families of manifold with corners}
\author{Florian THIRY}

\newcommand{\R}{\mathbb{R}}

\newcommand{\Z}{\mathbb{Z}}
\newcommand{\N}{\mathbb{N}}

\newcommand{\calK}{\mathcal{K}}

\newcommand{\calU}{\mathcal{U}}

\newcommand{\arInj}{\ar@{^{(}->}}
\newcommand{\arSurj}{\ar@{->>}}
\newcommand{\arEgal}{\ar@{=}}

\newcommand{\bigsqcupsim}{\underset{\sim}{\bigsqcup}}

\def\subar[#1]{\ar@{}[#1]|-*[@]{\subseteq}}
\def\supar[#1]{\ar@{}[#1]|-*[@]{\supseteq}}

\newtheorem{thm}{Theorem}[section]
\newtheorem{prop}[thm]{Proposition}
\newtheorem{lem}[thm]{Lemma}

\newtheorem{defi}{Definition}
\newtheorem{rem}{Remark}
\newtheorem{ex}{Example}

\bibliography{biblio.bib}

\begin{document}

\begin{center}
{\LARGE Geometric obstructions to fully ellipticity \\ for families of manifolds with corners \par}
\vspace{0.3cm}
{\normalsize Florian Thiry \footnote{Univ Toulouse, CNRS, IMT, Toulouse, France. Email address: \texttt{florian.thiry@math.univ-toulouse.fr}}} \\
\vspace{0.2cm}

\end{center}

\vspace{0.4cm}

\begin{center}
	\textbf{Abstract}
\end{center}

	\noindent In the present paper we study index theory for families of manifold with corners. In particular the K-theoretical obstruction for an elliptic operator to have a family (Fredholm) index. For codimension 1 corners (families with boundary) it gives an expected condition on indices associated to codimension 1 faces. The main theorem of this paper concern the codimension 2 case: in addition to the expected condition on indices associated to codimension 2 faces, there is an extra combinatoric condition implying the topology of the family base. This new condition is expressed in terms of conormal homology cycles with K-theoretical coefficients.
	\vspace{2cm}

\tableofcontents

\section*{Introduction}

It is an established fact that in the context of manifolds (resp. families of manifolds) with boundary, not every elliptic operator has a Fredholm index (resp. a family index). This fact dates back from Atiyah's work, for instance in \cite{AB65} Atiyah and Bott studied the so-called Atiyah Patodi Singer conditions (APS) in which some elliptic operators are shown to be Fredholm despite the boundary, this include the Dirac operators. In \cite{BC90a} and \cite{BC90b} Bismut and Cheeger study families of Dirac operators on manifolds with boundary, using tools like cone method and Levi-Civita superconnections, they were able to compute integral formulas for Chern caracters under (APS) conditions. This work was extended in \cite{MP97} by Melrose and Piazza where they proved an index theorem for families of Dirac operators on manifolds with boundary, this result has been extended to every Fredholm operators using groupoids machinery by Carrillo Rouse, Lescure and Monthubert \cite{RLM13}. Melrose and Piazza in their work also realised that the vanishing of a K-theoretical index map over the boundary allows to define (APS) conditions and then to recover Fredholm operators. This fact translates a K-theoretical obstructions for elliptic operators to be Fredholm when the manifold has a boundary. Monthubert and Nistor using groupoids technics pushed that study showing in a more general context (manifold with corners) that an elliptic operator can be perturbated into a Fredholm one if and only if a certain index map vanishes over the boundary \cite{MN11}. The vanishing of this "boundary index" is also discussed in \cite{DS19} or in \cite{NSS08} and \cite{NSS10}. In \cite{CRL18} and \cite{CRL25} Carrillo Rouse and Lescure compute geometrically this boundary index map for low codimension manifold with corners. Moreover in \cite{CRLV21} via another strategy, Carrillo Rouse, Lescure and Velasquez compute it rationally (i.e up to torsion) for every manifold with corners. \\ \\
Let's denote $Ind_{\partial}$ the boundary map, and $B$ a smooth manifold. The trivial manifold with embedded corners family over $B$, namely the cartesian product $B \times Y$, is itself endowed with embedded corners, then the Carrillo Rouse and Lescure work still apply. In the present paper we propose to extend their results for a general family $\pi : X \longrightarrow B$ whose typical fiber is a manifold with embedded corners $Y$. We restrict ourselves to cases where $X$ is itself a codimension $1$ and $2$ manifold with embedded corners. \\ For $Codim(X) = 1$, we explicitely compute both the obstruction space $K^0(\Gamma_{b, \mathfrak{f}}(\partial X)) \cong K^1(B)^{\# \mathcal{F}_1(X)}$ (i.e $Ind_{\partial}$ codomain) and the obstruction map $Ind_{\partial}$.

\begin{restatable*}[Vanishing of the boundary index in codimension 1]{thm}{premierCodThm}
	For $X \longrightarrow B$ a codimension 1 family of manifold with embedded corners, $[\sigma_T]_0 \in K^0(\Gamma^{tan}_{b, \mathfrak{f}}(X))$, the following are equivalent: \\
	$Ind_{\partial}([\sigma_T]_0) = 0$ ssi $\forall g \in \mathcal{F}_1(X), Ind_1^g([{\sigma_T}_{\vert g}]_0) = 0$ in $K^1(B)$.
\end{restatable*}
Each codimension 1 face is a smooth family of manifolds, then the restriction of the operator to it has a family index valued in $K^1(B)$, topologically computable via Atiyah-Singer theorem for families (e.g see \cite[thm 2.2.1]{CR23}). In our theorem $Ind_1^g$ denotes the suspension of this index for the codimension 1 face $g$. \\

In codimension $2$ we were able to make the obstruction space fit in the middle of a short exact sequence where left and right terms are explicitely computable in terms of conormal homology groups with K-theoretic coefficients:

\begin{restatable*}[Obstruction space estimation in codimension 2]{thm}{secondObstrctSpaceThm}
	The obstruction space $K_0(A_2/A_0) = K^0(G_F)$ fits in the following short exact sequence:
	\[
		0 \longrightarrow
		H_1^{pcn}(X, X_0 ; K^1(B)) \longrightarrow
		K_0(A_2/A_0) \longrightarrow
		H_0^{pcn}(X, X_0 ; K^0(B)) \longrightarrow 0.
	\]
\end{restatable*}

Then we caracterised the vanishing of the boundary map $Ind_{\partial}$. An extra geometrical condition appear in terms of conormal cycle, this is the main theorem of this paper:

\begin{restatable*}[Vanishing of the boundary index in codimension 2]{thm}{secondCodThm}
	Let $[\sigma_T]_0 \in K^0(\Gamma^{tan}_{b, \mathfrak{f}}(X))$. \\
	\[
		Ind_{\partial}([\sigma_T]_0) = 0 
		\Leftrightarrow
		\left|
		\begin{array}{l}
			\forall f \in \mathcal{F}_2(X), Ind_2^f([{\sigma_T}_{\vert f}]_0) = 0 \\
			\text{and the induced $\mathcal{F}_2(X)$-relative non commutative symbol $\overset{\circ}{\sigma_T}$ is such that} \\
			\left(Ind_{1, nc}^{\overline{g}}([\overset{\circ}{\sigma_T}_{\vert \overline{g}}]_0) \right)_{g \in \mathcal{F}_1(X)} \,\,\, \text{vanishes in $H_1^{pcn}(X, X_0 ; K^1(B))$}.
		\end{array}
		\right.
	\]
\end{restatable*}
Here $Ind_2^f$ denotes the 2-suspended family index of the smooth family $f$, and $Ind_{1, nc}^{\overline{g}}$ the suspension of the Fredholm family index on $\overline{g}$ such as defined in \cite{RLM13} for manifold with boundary. \\

Our approach behaving like a spectral sequence, to continue with diagram chasing we need existance of sections. The main obstruction to their existance lies in torsion. In \cite{CRL18} and \cite{CRL25} authors show that in their case (i.e $B = \{*\})$ conormal homology and the associated K-theory groups are torsion free when $Codim(X) \leqslant 2$. \\
In \cite{SV22} it is shown that torsion could appear in homology groups $H_*^{cn}(X) $ when $B = \{*\}$ and $Codim(X) \geqslant 3$. It implies that the strategy of the present paper cannot in general be extended for codimension bigger or equal to 3. \\

Another strategy could allows us to produce a rationnal computation (up to torsion) in every codimension using Connes-Thom isomorphisms and Chern caracter as in \cite{CRLV21}. This will be done in a next paper. \\ \\

In section 1 we introduce a $C^*$-algebraic framework to express a K-theoretical obstruction in a general abstract context. Then we apply it to pseudodifferential calculus to elaborate a relative calculus where the obstruction mentionned corresponds to the ability for an elliptic operator to be perturbated into a fully elliptic one, i.e to the boundary index. \\

In section 2 we introduce manifolds with embedded corners and their automorphism. Choosing those last as cocycles we define families of such manifolds and we give some examples. \\

In section 3 from such a family $X$ we introduce a groupoid analogue to Monthubert Puff groupoid denoted $\Gamma_{b, \mathfrak{f}}(X)$ and his canonical filtration of $C^*$-algebras $(A_k)_{k=0}^d$. We describe his differential structure and some local diffeomorphisms. Then we use explicit computations of $K_*(\frac{A_{k}}{A_{k-1}})$ to highlight link between conormal homology and the K-theoretical connections of the sequence $0 \rightarrow \frac{A_{k-1}}{A_{k-2}} \rightarrow \frac{A_{k}}{A_{k-2}} \rightarrow \frac{A_{k}}{A_{k-1}} \rightarrow 0$. \\

In section 4 we prove the announced results using comparison between K-theory and conormal homology.

\bigskip 

\noindent \textbf{Acknowledgements.} I would like to thank my advisor Paulo Carrillo Rouse for his guidance through this project. The author is funded by an EUR MINT PhD fellowship.

\nocite{*}

\newpage


\section{Obstruction and diagonal index}

\subsection{$C^*$ algebraic obstruction theorem}
For this section we set a $C^*-$algebra diagram as follows:
\[
    \xymatrix{& 0 \ar[d] & 0 \ar[d] & 0 \ar[d] & \\
    0 \ar[r] & A_1 \ar[r] \ar[d] & A_2 \ar[r]^r \ar[d] & A_3 \ar[r] \ar[d] & 0 \\
    0 \ar[r] & B_1 \ar[r] \ar[d] & B_2 \ar[r]^{\eta} \ar[d]^{\sigma} & B_3 \ar[r] \ar[d] & 0 \\
    0 \ar[r] & C_1 \ar[r] \ar[d] & C_2 \ar[r] \ar[d] & C_3 \ar[r] \ar[d] & 0 \\
    & 0 & 0 & 0 &}
\]
where algebras $B_1, B_2, B_3, C_1, C_2$ and $C_3$ are unitals, and the morphism between them are unit preserving. We also suppose every square to be commutative, every columns and rows exact. \\
The K-theory connection map of the second column $K_1(C_2) \longrightarrow K_0(A_2)$ will be denoted $\delta$. \\ \\
From such a diagram we can produce a short exact sequence,
\[
    \xymatrix{0 \ar[r] & A_1 \ar[r] & B_2 \ar[r]^{ \! \! \! \! \! \! \! \! \! \! \sigma \oplus \eta} & C_2 \underset{C_3}{\oplus} B_3 \ar[r] & 0}.
\]
We call it the $\textit{diagonal sequence}$. Usually we denote $\sigma \oplus \eta := \sigma_{diag}$, $\eta$ being understood. \\ \\
Our funding result is a link between $\sigma$-invertibility and $\sigma_{diag}$-invertibility. To prove it we recall a technical lemma:
\begin{lem}\label{lem:unitaryLift}
    Let $\varphi : A \longrightarrow B$ a surjective unit preserving morphism of unital $C^*$-algebras. \\
    Then $\varphi(\calU_0(A)) = \calU_0(B)$, with $\calU_0$ being the set of unitary elements homotopic to identity.
\end{lem}
This result is a consequence of continuous functionnal calculus. \\ \\
We can now state our funding result.
\begin{thm}
Let $T \in B_2$ an element $\sigma$-invertible (i.e $ \sigma(T)$ is invertible in $C_2$), the following propositions are equivalent:
\begin{enumerate}
    \item $K_0(r)(\delta [\sigma(T)]_1) = 0$
    \item $\exists N > 0, \exists T' \in M_N(B_2)$ with $\sigma(T')$ invertible, $[\sigma(T)]_1 = [\sigma(T')]_1$ and $\eta(T')$ invertible.
\end{enumerate}
    In particular $\sigma_{diag}(T')$ will be invertible in that case.
\end{thm}
The map $K_0(r) \circ \delta$, that we denote $Ind_{\partial}$, then encode an obstruction. The obstruction for a $\sigma$-invertible element to be pertubated into a $\sigma_{diag}$-invertible element.

\begin{proof}
    First, if $\sigma(T)$ is invertible, using the polar decompsition $\sigma(T) = \vert \sigma(T) \vert \omega(\sigma(T))$, and surjectivity of $\sigma$, there exists $V \in B_2$ such that $[\sigma(T)]_1 = [\omega(\sigma(T))]_1 = [\sigma(V)]_1$, $\sigma(V)$ being a unitary element. \\
    Then it is enough to show the theorem for $\sigma$-unitary elements (i.e elements $T \in B_2$ such that $\sigma(T)$ is unitary in $C_2$). \\ \\
    The implication $\underline{ (2) \Rightarrow (1)}$ is straightforward using squares commmutativity and naturality of connecting morphisms. We do focus on the reciprocal. \\ \\
	 $\underline{ (1) \Rightarrow (2)}$:
	 We need to describe the last diagram more precisely, we add the following notations:
\[\begin{tikzcd}[sep=small]
	& 0 & 0 & 0 \\
	0 & {A_1} & {A_2} & {A_3} & 0 \\
	0 & {B_1} & {B_2} & {B_3} & 0 \\
	0 & {C_1} & {C_2} & {C_3} & 0 \\
	& 0 & 0 & 0,
	\arrow[from=1-2, to=2-2]
	\arrow[from=1-3, to=2-3]
	\arrow[from=1-4, to=2-4]
	\arrow[from=2-1, to=2-2]
	\arrow[from=2-2, to=2-3]
	\arrow[from=2-2, to=3-2]
	\arrow["r", from=2-3, to=2-4]
	\arrow[from=2-3, to=3-3]
	\arrow[from=2-4, to=2-5]
	\arrow[from=2-4, to=3-4]
	\arrow[from=3-1, to=3-2]
	\arrow["{\mu_B}", from=3-2, to=3-3]
	\arrow["\alpha", from=3-2, to=4-2]
	\arrow["\eta", from=3-3, to=3-4]
	\arrow["\sigma", from=3-3, to=4-3]
	\arrow[from=3-4, to=3-5]
	\arrow["\gamma", from=3-4, to=4-4]
	\arrow[from=4-1, to=4-2]
	\arrow["{\mu_C}", from=4-2, to=4-3]
	\arrow[from=4-2, to=5-2]
	\arrow["{\eta_C}", from=4-3, to=4-4]
	\arrow[from=4-3, to=5-3]
	\arrow[from=4-4, to=4-5]
	\arrow[from=4-4, to=5-4]
\end{tikzcd}\]
and the connection map of the third column $K_1(C_3) \rightarrow K_0(A_3)$ is denoted $\delta_3$. \\ \\
	 Let $T \in B_2$ a unitary element such that $K_0(r)(\delta [\sigma(T)]_1) = 0$. Using connections naturality, we have: $0 = K_0(r)(\delta [\sigma(T)]_1) = \delta_3 K_1(\eta_C)([\sigma(T)]_1)$. \\
	 Then exactness of the 6 term exact sequence associated to the third column gives us a lift through $K_1(\gamma)$. The map $\eta$ being surjective it could be chosen a $V \in M_{\infty}(B_2)$ such that $\eta(V)$ is unitary. For such a $V$ we have:
	 \[
	 	K_1(\gamma)([\eta(V)]_1) = K_1(\eta_C)([\sigma(T)]_1).
	 \]
	 Using this equality, we can write $[\eta_C(\sigma(T))]_1 = K_1(\eta_C)([\sigma(T)]_1) = K_1(\gamma)([\eta(V)]_1) = [\gamma \circ \eta(V)]_1 = [\eta_C(\sigma(V))]_1$, which means that $\eta_C(\sigma(T))$ and $\eta_C(\sigma(V))$ are stably homotopic. \\ \\
	 Multiplying this homotopy by the stabilization of $\eta_C(\sigma(V))^{-1} = \eta_C(\sigma(V))^* = \eta_C(\sigma(V^*))$ we get an homotopy from $\eta_C \circ \sigma((T \oplus 1_N)(V^* \oplus 1_M)))$ to the unit ($N, M \in \N$). \\ \\
	 From lemma \ref{lem:unitaryLift}, there exists a unitary (homotopic to the unit) lift $W \in \calU_0(M_{\infty}(C_2))$:
	 \[
	 	\eta_C \circ \sigma((T \oplus 1_N)(V^* \oplus 1_M)) = \eta_C(W).
	 \]
	 Exactness ($Ker \, \eta_C = Im \, \mu_C$), $\alpha$ surjectivity and the commutativity ($\mu_C \circ \alpha = \sigma \circ \mu_B$) gives us another lift $b_1 \in M_{\infty}(B_1)$ such that: $\sigma(\mu_B(b_1)) + W = \sigma((T \oplus 1_N)(V^* \oplus 1_M)))$. \\ \\
	 From this expression, we extract $W^* \sigma(T \oplus 1_N) = \sigma((V^* \oplus 1_M) - \mu_B(b_1^*)(T \oplus 1_N))$. \\ \\
	 Then we set $\chi_T := \mu_B(b_1^*)(T \oplus 1_N)$, and $T' := (V^* \oplus 1_M) - \chi_T$ in such a way that we now have:
	 \[
	 	W^*\sigma(T \oplus 1_N) = \sigma((V^* \oplus 1_M) - \chi_T) = \sigma(T').
	 \] \\
	 This $T'$ is indeed the one we are looking for:
	 \begin{enumerate}
	 	\item $\sigma(T') = W^* \sigma(T \oplus 1_N)$ which is indeed a unitary element
	 	\item The element $W^*$ is unitary and homotopic to the identity, then $[W^*]_1 = 0$, which implies that $[\sigma(T')]_1 = [W^*]_1 + [\sigma(T \oplus 1_N)]_1 = [\sigma(T)]_1$
	 	\item Finally, because $\eta(\chi_T) = \eta(\mu_B(b_1^*)(T \oplus 1_N)) = \big(\eta \circ \mu_B(b_1^*)\big) \big(\eta(T \oplus 1_N)\big) = 0$, then $\eta(T') = \eta((V^* \oplus 1_M) - \chi_T) = \eta(V^* \oplus 1_M) = \eta(V^*) \oplus 1_M$ which is indeed a unitary element.
	 \end{enumerate}
\end{proof}

\subsection{Relative ellipticity}
For this section we will use the words \textit{singular manifolds} and \textit{non singular manifolds} without defining them. This is because the following discussion could be adapted in many contexts. In the next sections we will make a precise meaning of those notions to launch some computations. For now you can think of a \textit{singular manifold} as a non closed one, a \textit{non singular manifold} as a closed manifold. \\ \\
We recall that a Lie groupoid $\xymatrix{G \dar[r] & G^{(0)}}$ comes with a pseudodifferential calculus which allows us to write the sequence
\[
\xymatrix{0 \ar[r] & C^*(G) \ar[r] & \overline{\psi^0}(G) \ar[r]^{\sigma} & C_0(S^*G) \ar[r] & 0},
\]
where $\sigma$ denote the operator symbol. See \cite{NWX99} or \cite{DS19} for more details. \\
We suppose such a smooth Lie groupoid associated to a singular manifold in a way that we have a saturated decomposition in two smooth Lie groupoids:
\[
\label{ouvertferme}
\xymatrix{
G \dar[d] & = & G_U \dar[d] & \sqcup & G_F \dar[d] \\ G^{(0)} & = & U & \sqcup & F,}
\]
where $U$ is an open subset of $G^{(0)}$ and $F$ a closed one. \\
In particular this gives a short exact sequence (see for example \cite{Wil19}):
\[\begin{tikzcd}
	0 & {C^*(G_U)} & {C^*(G)} & {C^*(G_F)} & 0.
	\arrow[from=1-1, to=1-2]
	\arrow[from=1-2, to=1-3]
	\arrow[from=1-3, to=1-4]
	\arrow[from=1-4, to=1-5]
\end{tikzcd}\]
Here the idea is to think of $G_U$ as describing the non singular part of the manifold, and $G_F$ the singular one. Each of those groupoids have their own pseudodifferential calculus which are related, forming the following commutative diagram: \label{gauffre}
    \[
    \xymatrix{& 0 \ar[d] & 0 \ar[d] & 0 \ar[d] & \\
    0 \ar[r] & C^*(G_U) \ar[r] \ar[d] & C^*(G) \ar[r]^{r} \ar[d] & C^*(G_F) \ar[r] \ar[d] & 0 \\
    0 \ar[r] & \overline{\psi^0}(G_U) \ar[r] \ar[d] & \overline{\psi^0}(G) \ar[r] \ar[d]^{\sigma} & \overline{\psi^0}(G_F) \ar[r] \ar[d] & 0 \\
    0 \ar[r] & C_0(S^*G_U) \ar[r] \ar[d] & C_0(S^*G) \ar[r] \ar[d] & C_0(S^*G_F) \ar[r] \ar[d] & 0\\
    & 0 & 0 & 0 &.}
    \]
    
In such a context, recalling that an operator $T \in \overline{\psi^0}(G)$ is called \textit{elliptic on $G$} if $\sigma(T)$ is invertible in $C_0(S^*G)$, we define now a relative notion of ellipticity.

\begin{defi}\label{suiteDiag}
From a groupoid $G = G_U \sqcup G_F$ as above, we can define the diagonal sequence
\[
	\xymatrix{ 0 \ar[r] & C^*(G_U) \ar[r] & \overline{\psi^0}(G) \ar[r]^{\sigma_{diag}} &  A_F \ar[r] & 0},
\]
where $A_F = C_0(S^*G) \underset{C_0(S^*G_F)}{\oplus} \overline{\psi^0}(G_F)$. \\
An elliptic element which is moreover $\sigma_{diag}$ invertible is called \textit{fully elliptic} on $G$ (relatively to $U$ and $F$).
\end{defi}
Using that terminology, we can restate our $C^*$-algebraic result in the framework of elliptic operators on a groupoid with a decomposition:
\begin{thm}
For $T \in \overline{\psi^0}(G)$ an elliptic operator on $G$, the following are equivalent:
	\begin{enumerate}
		\item $Ind_{\partial}([\sigma(T)]_1) = 0$
		\item there exists $T'$ in $\overline{\psi^0}(G)$ fully elliptic, with $[\sigma(T)]_1 = [\sigma(T')]_1$.
	\end{enumerate}
\end{thm}

\begin{rem}\label{rk:symbVersTgt}
	For our purposes we will need to realise the map $Ind_{\partial}$ in a more geometrical way.
	To this end we recall that from a Lie groupoid $\xymatrix{G \dar[r] & G^{(0)}}$, we can define its tangent groupoid $\xymatrix{G^{tan} \dar[r] & G^{(0)} \times [0, 1]}$, $G^{tan}$ being the restriction above $G^{(0)} \times [0, 1]$ of the deformation to the normal cone of the manifold $G$ with respect to $G^{(0)}$ ($G^{(0)}$ seen as a submanifold of $G$ using the canonical embedding). \\
	
	Then using this last notion, from an elliptic operator $T \in \overline{\psi^0}(G)$ we can associate a class $[\sigma_T]_0 \in K_0(C^*(G^{tan}))$ in such a way that $Ind_{\partial}([\sigma(T)]_1) = r \circ e_1([\sigma_T]_0)$ where $e_1$ is defined as the restriction to the sub groupoid : $K_0(C^*(G^{tan})) \overset{e_1}{\rightarrow} K_0(C^*(G))$. Under this extra construction, the boundary index consists only in an evaluation to $G^{(0)} \times \{1\}$.
\[\begin{tikzcd}[ampersand replacement=\&]
	\& {K_0(C^*(G))} \\
	\&\&\& {K_0(G^{tan})} \& {[\sigma_T]_0} \\
	{[\sigma(T)]_1} \& {K_1(C_0(S^*G))} \&\& {K_0(AG)} \& {}
	\arrow["{e_1}"', from=2-4, to=1-2]
	\arrow["{e_0 \cong}", from=2-4, to=3-4]
	\arrow["\ni"{description}, draw=none, from=2-5, to=2-4]
	\arrow["\in"{description}, draw=none, from=3-1, to=3-2]
	\arrow[from=3-2, to=1-2]
	\arrow[from=3-2, to=3-4]
\end{tikzcd}\]
Both maps $Ind_{\partial}$ and $r_{\partial} \circ Ind^X := r \circ e_1$ will be called \textit{boundary index}.
\end{rem}

In this paper, this is this last expression we will use to caracterize the vanishing of the boundary index $Ind_{\partial}$ and then the ability for an elliptic operator to be perturbated into a fully elliptic one.

\begin{rem} \label{grpdFred}
	We can push this relative approach to produce more index maps as follows. \\
	For $G = G_U \bigsqcup G_F$ a saturated decomposition as before, we can define another saturated decomposition : $G^{tan} = G^{Fred} \bigsqcup G_F$, \, \, with $G^{Fred} := G^{tan} \setminus (G_F \times \{1\})$. \\ \\
	Then the restriction $G^{tan} \overset{e_1}{\rightarrow} G$ is compatible with the saturated decomposition in the sense that it restricts to
$
\left|
\begin{array}{l}
G^{Fred} \overset{e_1}{\rightarrow} G_U \\
G_F \overset{Id}{\rightarrow} G_F
\end{array}
\right.
$.
This produces a morphism between the associated diagonal sequences (Def \ref{suiteDiag}): \\
\[\begin{tikzcd}[ampersand replacement=\&]
	0 \& {C^*(G^{Fred})} \& {\overline{\psi^0}(G^{tan})} \& {A_F^{[0, 1]}} \& 0 \\
	0 \& {C^*(G_U)} \& {\overline{\psi^0}(G)} \& {A_F} \& 0,
	\arrow[from=1-1, to=1-2]
	\arrow[from=1-2, to=1-3]
	\arrow["{e_1}"', from=1-2, to=2-2]
	\arrow[from=1-3, to=1-4]
	\arrow["{e_1}"', from=1-3, to=2-3]
	\arrow[from=1-4, to=1-5]
	\arrow["{e_1}"', from=1-4, to=2-4]
	\arrow[from=2-1, to=2-2]
	\arrow[from=2-2, to=2-3]
	\arrow[from=2-3, to=2-4]
	\arrow[from=2-4, to=2-5]
\end{tikzcd}\]
where $A_F^{[0, 1]} := C_0(S^*G \times [0, 1]) \underset{C_0(S^*G_F)}{\bigoplus} \overline{\psi^0}(G_F)$ and $A_F := C_0(S^*G) \underset{C_0(S^*G_F)}{\bigoplus} \overline{\psi^0}(G_F)$. \\ \\
Noticing that $C_0(S^*G \times [0, 1[)$ is a contractible $C^*$-algebra, we have that $e_1 : A_F^{[0, 1]} \rightarrow A_F$ induces an isomorphism in K-theory. Denoting the upper row connection map $\delta^{tan}$, and noticing that the lower connection map $\delta^{diag}$ is the connection morphism of the diagonal sequence (def \ref{suiteDiag}) we have:
\[\begin{tikzcd}[ampersand replacement=\&]
	{K_1(A_F^{[0, 1]})} \& {K_0(C^*(G^{Fred}))} \\
	{K_1(A_F)} \& {K_0(C^*(G_U)).}
	\arrow["{\delta^{tan}}", from=1-1, to=1-2]
	\arrow["\cong"', from=1-1, to=2-1]
	\arrow["{e_1}"', from=1-2, to=2-2]
	\arrow["\delta^{diag}", from=2-1, to=2-2]
\end{tikzcd}\]
\end{rem}
Then if we start from a fully elliptic operator $T \in \overline{\psi^0}(G)$, it defines $[\sigma_{diag}(T)]_1 \in K_1(A_F)$. The associated diagonal index $\delta^{diag} [\sigma_{diag}(T)]_1$ is then computed through $G^{Fred} \overset{e_1}{\rightarrow} G_U$. \\
In particular if $G_U$ is the pair groupoid (resp. fibered pair groupoid) then it computes Fredholm indices (resp. family indices). \\
For example when in section \ref{indiceNc} we define $Ind_{nc}^X : K_0(C^*(\Gamma^{Fred}_{b, \mathfrak{f}}(X))) \overset{e_1}{\longrightarrow} K_0(C^*(X \underset{B}{\times} X)) \cong K^0_{top}(B)$ the restriction to $\{1\}$ of the Fredholm groupoid corresponds to a family index of a certain operator.
\newpage

\section{Family of manifolds with corners}
In this section we define the geometric context of family of manifolds with embedded corners. We suppose the notion of manifold with corners and embedded corner to be known. \\
We first make precise what we may call an \textit{embedded corner manifold oriented automorphism}.
\begin{defi}
	Let $Y$ be a manifold with embedded corners and $(r_j)_{j=1}^m$ its defining functions ($\Sigma_j := r_j^{-1}(0)$). \\
	We call \textit{oriented automorphism} of $Y$ the restriction to $Y$ of a diffeomorphism $f : \widetilde{Y} \rightarrow \widetilde{Y}$ endowed with a permutation $\sigma \in S_N$ such that $f^{-1}(\Sigma_j) = \Sigma_{\sigma(j)}$ for every $j \in \llbracket 1, m \rrbracket$. \\
	The orientation condition means that $f^{-1}( \{  r_j > 0 \} ) = \{ r_{\sigma(j)} > 0 \}$ and $f^{-1}( \{  r_j < 0 \} ) = \{ r_{\sigma(j)} < 0 \}$ for every $j \in \llbracket 1, m \rrbracket$. \\ \\
	The set of oriented automorphisms form a group denoted $Puff_+(Y)$.
\end{defi}

We will refer to the permutation $\sigma$ as \textit{the combinatoric} of $f$

\begin{rem}
	In such a context, an oriented automorphism $f \in Puff_+(Y)$ defines maps of pairs $f : (\widetilde{Y}, \Sigma_{\sigma(j)}) \rightarrow (\widetilde{Y}, \Sigma_j)$. This allows us to define oriented morphisms between the normal bundles of such pairs.
	Moreover it induces a smooth diffeomorphism $f : \overset{\circ}{Y} \rightarrow \overset{\circ}{Y}$.
\end{rem}
Using this automorphism notion, we can define a family of embedded corners manifold as follows:
\begin{defi}
	    Let $B$ a smooth manifold, $(\widetilde{Y}, (r_i)_{i=1}^m)$ an embedded corner manifold with its defining functions. 
        We call \textit{family of embedded corner manifold over $B$ with fiber $Y$} the data consisting of a smooth surjective submersion $\pi : \widetilde{X} \longrightarrow B$ such that :
    \begin{enumerate}
        \item We can cover $B$ with trivializing open subsets $U_{\alpha}$ such that the following diagram commutes : 
        $\xymatrix{\pi^{-1}(U_{\alpha}) \ar[r]_{\cong}^{\varphi_{\alpha}} \ar[d] & U_{\alpha} \times \widetilde{Y} \ar[ld]^{pr_1} \\ U_{\alpha}}$, with $\varphi_{\alpha}$ a smooth diffeomorphism.
        \item Transitions
        $\xymatrix{
        U_{\alpha \beta} \times \widetilde{Y} \ar@/^1.5pc/[rr]^{({\varphi_{\alpha}}_{\vert U_{\alpha \beta}}) \circ ({\varphi_{\beta}}_{\vert U_{\alpha \beta}})^{-1} := \, \varphi_{\alpha \beta}} \ar[dr]_{pr_1} &
        \pi^{-1}(U_{\alpha \beta}) \ar[l]_{\cong}^{{\varphi_{\beta}}_{\vert U_{\alpha \beta}}} \ar[r]^{\cong}_{{\varphi_{\alpha}}_{\vert U_{\alpha \beta}}} \ar[d]^{\pi} & 
        U_{\alpha \beta} \times \widetilde{Y} \ar[dl]^{pr_1} \\
        & U_{\alpha \beta} &
        }$
        are such that $\varphi_{\alpha \beta}(x, y) = (x, \Phi^x_{\alpha \beta}(y))$ where for each $x$ in $U_{\alpha}$ we have $\Phi^x_{\alpha \beta} \in Puff(\widetilde{Y})$.
    \end{enumerate}
\end{defi}

\begin{rem}{}
    Inside a given trivialisation, the transition morphisms $(\Phi^x_{\alpha \beta})_{x \in U_{\alpha \beta}} \subset Puff_+(Y) \subset Diff(\widetilde{Y})$ forms a continuous family of diffeomorphisms. \\
    The map
    \begin{tabular}{ccc}
        $Puff_+(Y)$ & $\longrightarrow$ & $\mathfrak{S}_m$ \\
         $f$ & $\mapsto$ & $\sigma_f$
    \end{tabular}
    being continous for the induced topology, it is constant on each connected component of $U_{\alpha \beta}$. We will simply denote $\sigma^x_{\alpha \beta} = \sigma_{\alpha \beta}$ supposing $U_{\alpha \beta}$ to be connected.
\end{rem}

\begin{prop}{}
    The total space of such a family $\pi : \widetilde{X} \longrightarrow B$ is a canonical extension of a manifold with corners (not embedded in general) denoted $X \subset \widetilde{X}$. Moreover the restriction $\pi_{\vert X} : X \longrightarrow B$ is still surjective.
\end{prop}

\begin{proof}
    To build $X$ it is enough to notice that in each trivialisation $U_{\alpha}$ there is a natural restriction of $U_{\alpha} \times \widetilde{Y}$ to $U_{\alpha} \times Y$. The transitions being oriented, they restrict as well. Then we get $X$ as a gluing of local restrictions using the restricted cocycles. \\
    Being a manifold with corners is a local property, which is clearly true in a trivialisation (the cross product of a manifold with corners and a smooth manifold is a manifold with corners of the same codimension). Then $X$ is a manifold with corners. \\
    The restriction $\pi_{\vert X} : X \longrightarrow B$ is clearly surjective.
\end{proof}

In practice some families of manifold with embedded corners are better behavied than others. Those we will focus on will be the one where the total corner manifold $X$ is itself endowed with embedded corners. Indeed those last ones may be endowed with tubular neighborhoods around their faces, which is not the case in general (see the one corner drop, or the following example \ref{counterEx}).
\begin{rem}\label{tubNeighborhood}
	Following \cite[thm 2.12]{PW19} under Whitney (B) condition it is possible to define tubular neighborhoods of intersected hypersurfaces $H_I$ ($H_I = \underset{i \in I}{\cap} H_i$, $p := \vert I \vert)$ in $\widetilde{X}$ with respect to the submersion $\pi$. \\ That is to say there exists a neighborhood $V$ containing $H_I$ and a diffeomorpism $\varphi : H_I \times \R^p \longrightarrow V$ with $\varphi_{\vert H_I \times \{0\}^p} = Id_{H_I}$ and $\pi(\varphi(x, \lambda)) = \pi(x)$. \\
	In other words we have a morphism of family of manifold with embedded corner:
\[\begin{tikzcd}[ampersand replacement=\&]
	{H_i \times \R} \&\& V \& {\widetilde{X}} \\
	\& B
	\arrow["\varphi", "\cong"', from=1-1, to=1-3]
	\arrow["{\pi \circ pr_1}"', from=1-1, to=2-2]
	\arrow["\pi", from=1-3, to=2-2]
	\arrow["\subset"{description}, draw=none, from=1-4, to=1-3]
\end{tikzcd}\]
in this sense we can say that locally an hypersurfaces intersection of a family with embedded corners is a "suspension" of a lower codimension family with embedded corners. \\

The Whitney (B) condition is a consequence of the regularity conditions on the defining functions: for $H_I =  \underset{i \in I}{\cap} \{ \rho_i = 0 \}$ we have $(\forall i \in I,  \rho_i(x) = 0) \Rightarrow ((d_x \rho_i)_{i \in I} $ independant family). See \cite[def p.25]{PW19} or its geometrical equivalent the $(b_S)$ condition in \cite{Tro78} and \cite{Tro79}. \\
\end{rem}

Here are few examples:

\begin{ex}{Trivial family} \\
	Every trivial family can be endowed with embedded corners. If $r : \widetilde{Y} \rightarrow \R$ is a defining function for $Y$, then the trivial family $X := Y \times B \rightarrow B$ can be endowed with $r \circ pr_1$.
\end{ex}

\begin{ex}{Möbius band - 1/2 twisted squared solid torus} \\
	We can set $Y := [-1, 1] \subset \R$ with the defining functions $r_1(y) := 1 - y, r_2(y) := 1 + y$, $B = S^1$. \\
	Then we get the Möbius band by gluing those $Y$ fibers over the circle with transition functions being identity except for one where we set the transition $(y \in Y \mapsto -y \in Y) \in Puff_+(Y)$. \\
	In that case, the Möbius band is globally embedded with one hypersurface because in every trivialisation $U$ we can set $\rho_U(x, y) := r_1(y)r_2(y) = 1-y^2$. This function being invarient under the transitions, it glues on a globally defining function $\rho$ satisfying the embedded corners axioms. \\ \\
	Following the same process with $Y = [-1, 1]^2$ the defining functions $(y_1, y_2) \mapsto 1 \pm y_1$, $((y_1, y_2) \mapsto 1 \pm y_2$ and one transition beign equal to $(y_1, y_2) \mapsto (-y_1, -y_2)) \in Puff_+(Y)$ we get the same announced example, also globally endowed with embedded corners.
\end{ex}

\begin{ex}{1/4 twisted squared solid torus} \label{counterEx} \\
	Setting the same fiber $Y$ as the 1/2 twisted squared solid torus case but changing the non identity transition function to be $((y_1, y_2) \mapsto (-y_2, y_1)) \in Puff_+(Y)$ we get the 1/4 twisted squared solid torus. \\
	This example cannot be globally endowed with embedded corners for combinatoric reasons. As a manifold with corners it only has one codimension 1 face, and one codimension 2 face. If such a defining hypersurface existed, then it would have to auto intersect to create the codimension 2 face. And then it would not be an hypersurface anymore.
\end{ex}

\section{Monthubert groupoids for families}
From now on, except if specifically stated, the families $\pi : \widetilde{X} \longrightarrow B$ we consider will have their total space $\widetilde{X}$ endowed with embedded corners.

\subsection{Monthubert groupoid for families}

To introduce the main tool of the following section, we state some notations: \\ \\
Let $p \in \N^*$, for a set $A$ and $I \in \mathcal{P}(\llbracket 1, p \rrbracket)$ we denote by $A_I^p$ the product $A_1 \times \dots \times A_p$ where $A_i := $
$\left\{
    \begin{array}{ll}
        \{ 0 \} & \text{ if } i \in I \\
        A & \text{ if } i \notin I
    \end{array}
\right.$.\\
Moreover, every tuple $I \in \mathcal{P}(\llbracket 1, p \rrbracket)$ will be considered \underline{sorted by ascending order}. \\

In the context of a tubular neighborhood containing a given face $f \in \mathcal{F}_p(X)$ and whose fibers are subsets of $\pi$-fibers $U \subset \widetilde{X}$  (see rk  \ref{tubNeighborhood} for the existence), restrictions to each face $g \in \mathcal{F}_k(X)$ ($k \leqslant p$) such that $f \subset \overline{g}$ provides family of manifolds with corners isomorphisms $U \cap g \cong f \times \R^{+*p}_I$ ($I \in \mathcal{P}(\llbracket 1, p \rrbracket, \vert I \vert = p$, $I$ sorted). Because knowing $f$ and the tubular neighborhood, such a face can be identified only using $I$, it seems reasonable to call such a face $g_I$. \\
In particular with these notations: $g_{\emptyset} = \overset{\circ}{X}$ and $g_{\{1, \dots, p \}} = f$. Moreover, using the hypersurfaces indices, we have $g_I \subset H_I$. \\

For a family of manifolds with embedded corners $\pi : \widetilde{X} \longrightarrow B$ of codimension $d$, we denote its associated Monthubert groupoid as: \\
$\Gamma_{b, \mathfrak{f}}(X): = 
\overset{\circ}{X} \underset{B}{\times} \overset{\circ}{X} \,
 \underset{\sim}{\bigsqcup} 
 \left(\underset{g \in \mathcal{F}_1}{\sqcup} (g \underset{B}{\times} g \times \R^{*+}) \right) \,
 \underset{\sim}{\bigsqcup}
 \dots
 \underset{\sim}{\bigsqcup}
 \left(\underset{g \in \mathcal{F}_d}{\sqcup} (g \underset{B}{\times} g \times \R^{*+d})\right)$

The symbols $\underset{\sim}{\sqcup}$ corresponds to a disjoint union endowed with a non trivial topological and differential structure. \\ \\
Its differential structure is described as follows: \\
Let $f \in \mathcal{F}_p(X)$, in a tubular neighborhood $U$ around $f$, we have
\begin{align*}
    & \underset{k=0}{\overset{p}{\bigsqcupsim}}
    \left(
    \underset{I \in \mathcal{P}(\llbracket 1,  p \rrbracket), \vert I \vert = k}{\bigsqcup} (g_I \cap U) \underset{B}{\times} (g_I \cap U) \times \R^{+p}_{\{ 1, \dots, p \} \setminus I}
    \right) \\
    = \: \, &
    f \underset{B}{\times} f \times
    \underset{k=0}{\overset{p}{\bigsqcupsim}}
    \left(
    \underset{I \in \mathcal{P}(\llbracket 1,  p \rrbracket), \vert I \vert = k}{\bigsqcup} \R^{+*p}_I \times \R^{+*p}_I \times \R^{+*p}_{\{ 1, \dots, p \} \setminus I}
    \right)
\end{align*}
each member of the union is topologically described using the following homeomorphism: \\ \\
\begin{tabular}{ccc}
     $\R^{+*p}_I \times \R^{+*p}_I \times \R^{+*p}_{\{1 \dots p \} \setminus I}$ & $\longrightarrow$ & $\R^{+*}_I \times \R^{+*p}$ \\
     $(x, y, \lambda)$ & $\mapsto$ & $(x, a^I(x, y, \lambda))$ 
\end{tabular}
with $a^I_i(x, y, \lambda) = $
$\left\{
    \begin{array}{ll}
        \lambda_i & \text{ if } i \in I \\
        \frac{y_i}{x_i} & \text{ if } i \notin I
    \end{array}
\right.$
\\
The others disjoint unions (the ones where closures are also disjoint) are endowed with the trivial disjoint union topology.

\begin{ex}
	For instance, in the case $p=2$ (one face $f$ defined as the intersection of the closure of exactly two other codimension 1 faces), the Monthubert groupoid is written locally as:
\[
	(f \underset{B}{\times} f) \times \Big(\underset{I \in \{\emptyset, \{1\}, \{2\}, \{1, 2\}\}}{\bigsqcup} \R^{+*2}_I \times \R^{+*2}_I \times \R^{+*2}_{\{ 1, 2 \} \setminus I} \Big),
\]
the euclidian part endowed with the chart:
\[
	\begin{array}{ccccc}
		(x_1, x_2), & (y_1, y_2), & (0, 0) & \longmapsto & (x_1, x_2, \frac{y_1}{x_1}, \frac{y_2}{x_2}) \\
		(0, x_2), & (0, y_2), & (\lambda_1, 0) & \longmapsto & (x_1, x_2, \lambda_1, \frac{y_2}{x_2}) \\
		(x_1, 0), & (y_1, 0), & (0, \lambda_2) & \longmapsto & (x_1, x_2, \frac{y_1}{x_1}, \lambda_2) \\
		(0, 0), & (0, 0), & (\lambda_1, \lambda_2) & \longmapsto & (x_1, x_2, \lambda_1, \lambda_2).
	\end{array}
\]
\end{ex}

The local Monthubert Groupoid can also be described as follows:
\begin{prop}
    The previous homeomorphism extend in a groupoid homeomorphism 
    \begin{equation}\label{homeo:grpdActions}
        \xymatrix{
        \underset{k=0}{\overset{p}{\bigsqcupsim}}
    \left(
        \underset{I \in \mathcal{P}(\llbracket 1,  p \rrbracket), \vert I \vert = k}{\bigsqcup} \R^{+*p}_I \times \R^{+*p}_I \times \R^{+*p}_{\{ 1, \dots, p \} \setminus I}
    \right)
    \ar[r]^{\quad \quad \quad \quad \quad \quad \quad \cong} \dar[d] &
    \R^{+p} \rtimes \R^{+*p} \dar[d] \\
    \R^{+p} \ar@{=}[r] & \R^{+p}
    }
    \end{equation}
    with the groupoid associated to the right action $\R^{+p} \curvearrowleft \R^{+*p}$ defined by $t * \lambda := t \lambda$ with $(t\lambda)_i := t_i \lambda_i$.
\end{prop}

\begin{rem}{}
    We can notice that the previous homeomorphism restricts to each member of the union : \\
    \[
        \R^{+*p}_I \times \R^{+*p}_I \times \R^{+*p}_{\{1, \dots, p\} \setminus I}
        \cong
        \R^{+*p}_I \rtimes \R^{+*p}.
    \]
    Doing so, we induce a topology on some closed subspaces of $\R^{+p} \rtimes \R^{+*p}$.
\end{rem}

Supposing the amenability of the Monthubert Groupoid, we define a filtration $\emptyset = A_{-1} \subset A_0 \subset \dots \subset A_d = C^*(\Gamma_{b, \mathfrak{f}}(X))$ of its full $C^*$-algebra by $A_i := C^*\left(\underset{k \leqslant i}{\bigsqcupsim} (\underset{g \in \mathcal{F}_k}{\sqcup} g \underset{B}{\times} g \times \R^{+*k})\right)$. \\
We notice that for $i \leqslant j$, $\frac{A_j}{A_i} \cong C^*\left(\underset{i < k \leqslant j}{\bigsqcupsim} (\underset{g \in \mathcal{F}_k}{\sqcup} g \underset{B}{\times} g \times \R^{+*k})\right)$.
\subsection{Connections computation}
We still start from a family of manifolds with embedded corners $\pi : \widetilde{X} \longrightarrow B$ of codimension $d$ and whose total space is endowed with an embedded corner structure. We still endow its Monthubert groupoid with the previous filtration $\emptyset = A_{-1} \subset A_0 \subset \dots \subset A_d = C^*(\Gamma_{b, \mathfrak{f}}(X))$. \\

Let $p \in \llbracket 1, d \rrbracket$. \\
The short exact sequence $0 \rightarrow \frac{A_{p-1}}{A_{p-2}} \rightarrow \frac{A_p}{A_{p-2}} \rightarrow \frac{A_p}{A_{p-1}} \rightarrow 0$ provides a 6 term exact sequence:

{\tiny
\[\begin{tikzcd}
	{K^1(\underset{g \in \mathcal{F}_{p-1}}{\bigsqcup} g \underset{B}{\times}g \times \R^{+*(p-1)})} & {K^1\left(\underset{g \in \mathcal{F}_{p-1}}{\bigsqcup} g \underset{B}{\times}g \times \R^{+*(p-1)} \quad \underset{\sim}{\bigsqcup} \quad\underset{f \in \mathcal{F}_p}{\bigsqcup} f \underset{B}{\times}f \times \R^{+*p}\right)} & {K^1(\underset{f \in \mathcal{F}_p}{\bigsqcup} f \underset{B}{\times}f \times \R^{+*p})} \\
	\\
	{K^0(\underset{f \in \mathcal{F}_p}{\bigsqcup} f \underset{B}{\times}f \times \R^{+*p})} & {K^0\left(\underset{g \in \mathcal{F}_{p-1}}{\bigsqcup} g \underset{B}{\times}g \times \R^{+*(p-1)} \quad \underset{\sim}{\bigsqcup} \quad\underset{f \in \mathcal{F}_p}{\bigsqcup} f \underset{B}{\times}f \times \R^{+*p}\right)} & {K^0(\underset{g \in \mathcal{F}_{p-1}}{\bigsqcup} g \underset{B}{\times}g \times \R^{+*(p-1)})}
	\arrow[from=1-1, to=1-2]
	\arrow[from=1-2, to=1-3]
	\arrow[from=1-3, to=3-3]
	\arrow[from=3-1, to=1-1]
	\arrow[from=3-2, to=3-1]
	\arrow[from=3-3, to=3-2]
\end{tikzcd}\]
}

We have $K^i(g \underset{B}{\times} g) \cong K^i(B)$ ($\xymatrix{g \underset{B}{\times} g \dar[r] & g}$ and $\xymatrix{B \dar[r] & B}$ are Morita equivalent), then using Bott periodicity, the connection maps can be written as $\delta^i : K^i(B)^{\# \mathcal{F}_p} \longrightarrow K^i(B)^{\# \mathcal{F}_{p-1}}$, $i \in \{0, 1 \}$. \\ \\
Our goal in this section will be to compute those maps coordinatewise using local arguments. Let $f \in \mathcal{F}_p(X)$. \\ \\

The restriction to any faces is natural in the following sense: \\
For any $g$ in $\mathcal{F}_{p-1}(X)$, the $g$-component of the range of $\delta^i_{\vert f} : K^i(B) \longrightarrow \underset{g' \in \mathcal{F}_{p-1}(X)}{\bigoplus} K^i(B)$ is the connection of the restricted sequence
{\scriptsize
\begin{equation}\label{diag:suite_fg}
\begin{tikzcd}
	0 & {C^*(g \underset{B}{\times}g \times \R^{+*(p-1)})} & {C^*(g \underset{B}{\times}g \times \R^{+*(p-1)} \quad \underset{\sim}{\bigsqcup} \quad f \underset{B}{\times}f \times \R^{+*p})} & {C^*(f \underset{B}{\times}f \times \R^{+*p})} & 0
	\arrow[from=1-1, to=1-2]
	\arrow[from=1-2, to=1-3]
	\arrow[from=1-3, to=1-4]
	\arrow[from=1-4, to=1-5]
\end{tikzcd}.
\end{equation}
}

The first case is the one where $f \not \subset \overline{g}$. In such a case the disjoint union involved is topologically trivial. \\
We get $C^*(g \underset{B}{\times}g \times \R^{+*(p-1)} \quad \underset{\sim}{\bigsqcup} \quad f \underset{B}{\times}f \times \R^{+*p}) = C^*(g \underset{B}{\times}g \times \R^{+*(p-1)} \quad \bigsqcup \quad f \underset{B}{\times}f \times \R^{+*p}) = C^*(g \underset{B}{\times}g \times \R^{+*(p-1)}) \bigoplus  C^*(f \underset{B}{\times}f \times \R^{+*p})$. \\ \\
Then the short exact sequence (\ref{diag:suite_fg}) actually splits, and then connections are zero maps in this case. \\

To understand what happens with the other faces, we need to look at our groupoid using a tubular neighborhood. Let $U \subset \widetilde{X}$ such a neighborhood around the face $f$ with respect to $\pi$-fibers (see rk \ref{tubNeighborhood}). Using the notation $U \cap g_I \cong f \times \R^{+*}_I$, $g_I \subset H_I$, and $\widehat{i} := \{1, \dots, p \} \setminus \{i\}$, the homeomorphism (\ref{homeo:grpdActions}) gives us a local computation: \\
\begin{align*}
    & (g_{\widehat{i}} \cap U) \underset{B}{\times} (g_{\widehat{i}} \cap U) \times \R^{+*(p-1)} \quad \underset{\sim}{\bigsqcup} \quad f \underset{B}{\times} f \times \R^{+*p} \\
    \underset{tub. neighb.}{\cong} &
    f \underset{B}{\times} f \times \R^{+*p}_{\widehat{i}} \times \R^{+*p}_{\widehat{i}} \times \R^{+*p}_{\{i\}} \quad \underset{\sim}{\bigsqcup} \quad f \underset{B}{\times} f \times \R^{+*p}_{\{1, \dots, p \}} \times \R^{+*p}_{\{1, \dots, p \}} \times \R^{+*p}_{\emptyset} \\
    \underset{\text{homeo } (\ref{homeo:grpdActions})}{\cong} &
    f \underset{B}{\times} f \times \R^{+*p}_{\widehat{i}} \rtimes \R^{+*p} \quad \bigsqcup \quad f \underset{B}{\times} f \times \R^{+*p}_{\{1, \dots, p \}} \rtimes \R^{+*p} \\
    = & f \underset{B}{\times} f \times (\R^{+*p}_{\widehat{i}} \bigsqcup \{0\}^p) \rtimes \R^{+*p} \\
    = & f \underset{B}{\times} f \times \R^{+p}_{\widehat{i}} \rtimes \R^{+*p}.
\end{align*} \\

Then we can summerize using the following diagram:
{\footnotesize
\[\text{\hspace{-6em}}\begin{tikzcd}[ampersand replacement=\&,column sep=tiny]
	0 \& {C^*(g_I \underset{B}{\times}g_I \times \R^{+*(p-1)})} \& {C^*(g_I \underset{B}{\times}g_I \times \R^{+*(p-1)} \quad \underset{\sim}{\bigsqcup} \quad f \underset{B}{\times}f \times \R^{+*p})} \& {C^*(f \underset{B}{\times}f \times \R^{+*p})} \& 0 \\
	0 \& {C^*((g_I \cap U) \underset{B}{\times}(g_I \cap U) \times \R^{+*(p-1)})} \& {C^*((g_I \cap U) \underset{B}{\times}(g_I \cap U) \times \R^{+*(p-1)} \quad \underset{\sim}{\bigsqcup} \quad f \underset{B}{\times}f \times \R^{+*p})} \& {C^*(f \underset{B}{\times}f \times \R^{+*p})} \& 0 \\
	0 \& {C^*(f \underset{B}{\times}f \times \R^{+*}_{\widehat{i}} \rtimes \R^{+*p})} \& {C^*(f \underset{B}{\times}f \times \R^+_{\widehat{i}} \rtimes \R^{+*p})} \& {C^*(f \underset{B}{\times}f \times \{0\}^p \rtimes \R^{+*p})} \& 0 \\
	0 \& {C^*(B \times \R^{+*}_{\widehat{i}} \rtimes \R^{+*p})} \& {C^*(B \times \R^+_{\widehat{i}} \rtimes \R^{+*p})} \& {C^*(B \times \{0\}^p \rtimes \R^{+*p})} \& 0, \\
	\\
	{}
	\arrow[from=1-1, to=1-2]
	\arrow[from=1-2, to=1-3]
	\arrow[from=1-3, to=1-4]
	\arrow[from=1-4, to=1-5]
	\arrow[from=2-1, to=2-2]
	\arrow["\iota", from=2-2, to=1-2]
	\arrow[from=2-2, to=2-3]
	\arrow[from=2-3, to=1-3]
	\arrow[from=2-3, to=2-4]
	\arrow["{=}"{description}, no head, from=2-4, to=1-4]
	\arrow[from=2-4, to=2-5]
	\arrow[from=3-1, to=3-2]
	\arrow["\cong"{description}, no head, from=3-2, to=2-2]
	\arrow[from=3-2, to=3-3]
	\arrow["T", from=3-2, to=4-2]
	\arrow["\cong"{description}, no head, from=3-3, to=2-3]
	\arrow[from=3-3, to=3-4]
	\arrow["T", from=3-3, to=4-3]
	\arrow["{=}"{description}, no head, from=3-4, to=2-4]
	\arrow[from=3-4, to=3-5]
	\arrow["T", from=3-4, to=4-4]
	\arrow[from=4-1, to=4-2]
	\arrow[from=4-2, to=4-3]
	\arrow[from=4-3, to=4-4]
	\arrow[from=4-4, to=4-5]
\end{tikzcd}\]
}
where the last short exact sequence morphism is obtained using integration along fibers: \\
\begin{tabular}{cccc}
	$T : $ & $C^*(f \underset{B}{\times} f)$ & $\longrightarrow$ & $C_0(B)$ \\
	& $\varphi$ & $\mapsto$ & $\Big( b \mapsto \int_{\pi^{-1}(B)}  \varphi(x, x) dx\Big)$
\end{tabular}
, which induce an isomorphism in K-theory. \\ \\
Adding the following lemma:
\begin{lem}
	The map $\iota : C^*((g_I \cap U) \underset{B}{\times} (g_I \cap U) \times \R^{+*(p-1)}) \longrightarrow C^*(g_I \underset{B}{\times} g_I \times \R^{+*(p-1)})$ induce an isomorphism in K-theory.
\end{lem}

\begin{proof}
	Up to suspension, it is enough to show this result for $p=1$. \\
	
	
It is straightforward to check the commutation of: 
\[\begin{tikzcd}[ampersand replacement=\&]
	{C^*(g_I \underset{B}{\times} g_I)} \& {C_0(B)} \\
	{C^*((g_I \cap U) \underset{B}{\times} (g_I \cap U))} \& {C_0(B)}
	\arrow["T", from=1-1, to=1-2]
	\arrow["\iota", from=2-1, to=1-1]
	\arrow["T", from=2-1, to=2-2]
	\arrow["Id"', from=2-2, to=1-2]
\end{tikzcd}\]
Then functoriality tells us that $K_*(\iota)$ is an isomorphism (which can be identified with $Id_{K^*(B)}$).
\end{proof}

Our problem is now reduced to the computation of the connection maps of the last sequence:
\[\begin{tikzcd}
	0 & {C^*(B \times \R^{+*p}_{\widehat{i}} \rtimes \R^{+*p})} & {C^*(B \times \R^{+p}_{\widehat{i}} \rtimes \R^{+*p})} & {C^*(B \times \{0\}^p \rtimes \R^{+*p})} & 0,
	\arrow[from=1-1, to=1-2]
	\arrow[from=1-2, to=1-3]
	\arrow[from=1-3, to=1-4]
	\arrow[from=1-4, to=1-5]
\end{tikzcd}\]$i \in \llbracket 1, p \rrbracket$. We denote $\alpha_i^*$, $*=0,1$ those two connections maps. \\

Rearranging terms $(x, \lambda) \mapsto (x_k, \lambda_k)_{k=1}^p$ gives an isomorphism 
\begin{equation}\label{homeo:rearranger}
	\R^{+p} \rtimes \R^{+*p} \cong (\R^+ \rtimes \R^{+*})^p,
\end{equation}
which for every $i \in \llbracket 1, p \rrbracket$ restricts to:
\[
	\R^{+p}_{\widehat{i}} \rtimes \R^{+*p} \cong \Big(\{0\} \rtimes  \R^{+*}\Big)^{i-1} \times \Big(\R^+ \rtimes \R^{+*}\Big) \times \Big(\{0\} \rtimes \R^{+*}\Big)^{p-i}.
\]
Then every connecting map will be deduced from each other via a factor permutation.

\begin{lem}
	There are isomorphisms identifying the connection maps $\alpha_1^*$ as $Id_{K^{\diamond}(B)}$ ($*, \diamond \in \{0, 1\}$ with $\diamond = *$ or $\diamond = 1-*$ depending on the parity with $p$). \\ More precisely, there are isomorphisms making this square commutative:
\[\begin{tikzcd}
	{K^{1-*}(B \times \{0\}^p \rtimes \R^{+*p})} & {K^*(B \times \R^{+*}_{\widehat{1}} \rtimes \R^{+*p})} \\
	{K^{\diamond}(B)} & {K^{\diamond}(B)}
	\arrow["{\alpha_1^*}", from=1-1, to=1-2]
	\arrow["\cong"{description}, no head, from=1-1, to=2-1]
	\arrow["\cong"{description}, no head, from=1-2, to=2-2]
	\arrow["{Id_{K^{\diamond}(B)}}", from=2-1, to=2-2]
\end{tikzcd}\]
\end{lem}

\begin{proof}
	For $p=1$, the sequence we are looking for connection maps can be written as:
\[\begin{tikzcd}
	0 & {C_0(B) \otimes C^*(\R^{+*} \rtimes \R^{+*})} & {C_0(B) \otimes C^*( \R^{+} \rtimes \R^{+*})} & {C_0(B) \otimes C^*( \{0\} \rtimes \R^{+*})} & 0,
	\arrow[from=1-1, to=1-2]
	\arrow[from=1-2, to=1-3]
	\arrow[from=1-3, to=1-4]
	\arrow[from=1-4, to=1-5]
\end{tikzcd}\]
where the maps are identity on the $C_0(B)$ components.
Using \cite[prop 4.7.6]{HR00} $C_0(B)$ being a nuclear $C^*$ algebra, the connections are of the form $\alpha_1^* = Id_{K^*(B)} \otimes \alpha$ where $\alpha$ is the connecting map of the sequence:
\[\begin{tikzcd}
	0 & {C^*(\R^{+*} \rtimes \R^{+*})} & {C^*( \R^{+} \rtimes \R^{+*})} & {C^*( \{0\} \rtimes \R^{+*})} & 0.
	\arrow[from=1-1, to=1-2]
	\arrow[from=1-2, to=1-3]
	\arrow[from=1-3, to=1-4]
	\arrow[from=1-4, to=1-5]
\end{tikzcd}\]
The groupoids $\R^{+*} \rtimes \R^{+*}$ and $\R^{+*} \times \R^{+*}$ being isomorphic, we have $C^*(\R^{+*} \rtimes \R^{+*}) \cong C^*(\R^{+*} \times \R^{+*}) \cong \calK(L^2(\R^{+*}))$ and $C^*(\{0\} \rtimes \R^{+*}) \cong C_0(\R^{+*})$ we can set:
\begin{itemize}
	\item $p \in \calK(L^2(\R^{+*}))$ a rank one projector
	\item $\beta \in K_1(C_0(\R^{+*}))$ such that  $K_1(C_0(\R^{+*})) \rightarrow K_0(\calK(L^2(\R^{+*})))$ maps $\beta$ on $[p]_0$
	\item The isomorphism $K_1(C_0(\R^{+*})) \cong \Z$ generated by $\beta \mapsto 1$
	\item The isomorphism $K_0(\calK(L^2(\R^{+*}))) \cong \Z$ generated by $[p]_0 \mapsto 1$
\end{itemize}
fiting them together in a diagram gives us:
{\begin{equation}\label{eq:isobetap}
\begin{tikzcd}
	{K_1(C^*( \{0\} \rtimes \R^{+*}))} & {K_0(C^*(\R^{+*} \rtimes \R^{+*}))} \\
	{K_1(C_0(\R^{+*}))} & {K_0(\calK(L^2(\R^{+*})))} \\
	\Z & \Z,
	\arrow["\alpha", from=1-1, to=1-2]
	\arrow["\cong"{description}, no head, from=1-1, to=2-1]
	\arrow["\cong"{description}, no head, from=1-2, to=2-2]
	\arrow["{\beta \mapsto [p]_0}"', from=2-1, to=2-2]
	\arrow["\cong"{description}, no head, from=2-1, to=3-1]
	\arrow["\cong"{description}, no head, from=2-2, to=3-2]
	\arrow["{Id_{\Z}}", dashed, from=3-1, to=3-2]
\end{tikzcd}
\end{equation}}
where the induced isomorphism is $Id_{\Z}$.
Then tensoring those spaces by $K^*(B) = K_*(C_0(B))$ at left and the maps with $Id_{K^*(B)}$ and using the Künneth formulas we get the following commutative square:
\[\begin{tikzcd}
	{K_{1-*}(C_0(B) \otimes C^*( \{0\} \rtimes \R^{+*}))} & {K_*(C_0(B) \otimes C^*(\R^{+*} \rtimes \R^{+*}))} \\
	{K^*(B)} & {K^*(B).}
	\arrow["{\alpha_1^*}", from=1-1, to=1-2]
	\arrow["\cong"{description}, no head, from=2-1, to=1-1]
	\arrow["{Id_{K^*(B)}}", from=2-1, to=2-2]
	\arrow["\cong"{description}, no head, from=2-2, to=1-2]
\end{tikzcd}\]

Then using the suspension functor $p-1$ times on this square (or tensoring by $C_0(\R^{+*p})$, which is the same), we get the announced result.
\end{proof}

The connections $\alpha_1^*$ being computed through isomorphisms, we deduce the other $\alpha_i^*$ through the same isomorphisms.

\begin{lem}
	For $i \in \llbracket 2, p \rrbracket$, let $\tau_i : \R^{+p}_{\widehat{i}} \rtimes \R^{+*p} \overset{\cong}{\rightarrow} \R^{+p}_{\widehat{i-1}} \rtimes \R^{+*p}$ the groupoid isomorphism consisting in the exchange of the $i^{th}$ and the $i-1^{th}$ factor in both terms. \\
	The map $\tau_i$ induce a morphism between the connecting maps as follows:
\[\begin{tikzcd}
	{K^{1-*}(B \times \{0\}^p \rtimes \R^{+*p})} & {K^*(B \times \R^{+*p}_{\widehat{i}} \rtimes \R^{+*p})} \\
	{K^{1-*}(B \times \{0\}^p \rtimes \R^{+*p})} & {K^*(B \times \R^{+*p}_{\widehat{i-1}} \rtimes \R^{+*p}).}
	\arrow["{\alpha_i^*}", from=1-1, to=1-2]
	\arrow["{K_{1-*}(\tau_i)}"', from=1-1, to=2-1]
	\arrow["{K_{*}(\tau_i)}", from=1-2, to=2-2]
	\arrow["{\alpha_{i-1}^*}"', from=2-1, to=2-2]
\end{tikzcd}\]
Using the same isomorphisms as in $(\ref{eq:isobetap})$ and the K-theory stability under tenoring by $\calK$, the vertical left (resp. the right) morphism is identified with $-Id_{K^{\diamond}(B)}$ (resp. $Id_{K^{\diamond}(B)}$). \\
Then $\alpha_i^*$ and $\alpha_{i-1}^*$ are identified with opposite morphisms from $K^{\diamond}(B)$ to $K^{\diamond}(B)$.
\end{lem}

\begin{proof}
	the isomorphism (\ref{homeo:rearranger}) restricts to:
\[
	\R^{+p}_{\widehat{i}} \rtimes \R^{+*p} \cong \Big(\{0\} \rtimes  \R^{+*}\Big)^{i-2} \times \Big(\{0\} \rtimes \R^{+*}\Big) \times \Big(\R^+ \rtimes \R^{+*}\Big) \times \Big(\{0\} \rtimes \R^{+*}\Big)^{p-i}\]

\[	
	\R^{+p}_{\widehat{i-1}} \rtimes \R^{+*p} \cong \Big(\{0\} \rtimes  \R^{+*}\Big)^{i-2} \times \Big(\R^+ \rtimes \R^{+*}\Big) \times \Big(\{0\} \rtimes \R^{+*}\Big) \times \Big(\{0\} \rtimes \R^{+*}\Big)^{p-i}.
\]
then up to tensoring by $C_0(\R^{+*}) = C^*(\{0\} \rtimes \R^{+*})$ $i-2$ times at left, $p-1$ times at right, and by $K^*(B)$ at left, we can suppose that $p=2$. For this proof, connection maps with those simplifying assumptions will be denoted $\widetilde{\alpha_1}$ and $\widetilde{\alpha_2}$. \\

The isomorphism $(\ref{homeo:rearranger})$ being understood, $\tau_2$ may be written as $\tau_2 : (\{0\} \rtimes \R^{+*}) \times (\R^+ \rtimes \R^{+*}) \overset{\cong}{\rightarrow} (\R^+ \rtimes \R^{+*}) \times (\{0\} \rtimes \R^{+*})$. It induces a short exact sequence morphism whose connection naturality rise a commutative square:
\[\begin{tikzcd}
	{K^0((\{0\} \rtimes \R^{+*}) \times (\{0\} \rtimes \R^{+*}))} & {K^1((\{0\} \rtimes \R^{+*}) \times (\R^{+*} \rtimes \R^{+*}))} \\
	{K^0((\{0\} \rtimes \R^{+*}) \times (\{0\} \rtimes \R^{+*}))} & {K^1((\R^{+*} \rtimes \R^{+*}) \times (\{0\} \rtimes \R^{+*}))}
	\arrow["{\widetilde{\alpha_2}}", from=1-1, to=1-2]
	\arrow["{\cong \tau_2}", from=1-1, to=2-1]
	\arrow["{\cong \tau_2}", from=1-2, to=2-2]
	\arrow["{\widetilde{\alpha_1}}"', from=2-1, to=2-2]
\end{tikzcd}\]
Using the suspended vertical isomorphisms on the upper half of $(\ref{eq:isobetap})$, $\widetilde{\alpha_1}$ and $\widetilde{\alpha_2}$ may be identfied with the maps $Id \otimes (\beta \mapsto [p]_0)$ and $(\beta \mapsto [p]_0) \otimes Id$. Pursuing the identifications using suspensions of the $(\ref{eq:isobetap})$ lower half isomorphisms, we get:
\[\begin{tikzcd}
	\Z & {K_0(C_0(\R^{+*2}))} &&& {K_1(C_0(\R^{+*}) \otimes \calK)} & \Z \\
	\Z & {K_0(C_0(\R^{+*2}))} &&& {K_1(\calK \otimes C_0(\R^{+*}))} & \Z
	\arrow["Id"{description}, bend left=15, from=1-1, to=1-6]
	\arrow[dashed, from=1-1, to=2-1]
	\arrow["\cong", from=1-2, to=1-1]
	\arrow["{Id \otimes (\beta \mapsto [p]_0)}", from=1-2, to=1-5]
	\arrow["{\tau_2}"', from=1-2, to=2-2]
	\arrow["\cong"', from=1-5, to=1-6]
	\arrow["{\tau_2}", from=1-5, to=2-5]
	\arrow[dashed, from=1-6, to=2-6]
	\arrow["Id"{description}, bend right=15, from=2-1, to=2-6]
	\arrow["\cong"', from=2-2, to=2-1]
	\arrow["{(\beta \mapsto [p]_0) \otimes Id}"', from=2-2, to=2-5]
	\arrow["\cong", from=2-5, to=2-6]
\end{tikzcd}\]
where the curved arrows are identity as suspension of $Id_{\Z}$ in $(\ref{eq:isobetap})$.

By construction K-theory of $C^*$-algebras is stable under tensorisation by $\calK$, then : $K_1(C_0(\R^{+*}) \otimes \calK) = K_1(\calK \otimes C_0(\R^{+*})) = K_1(C_0(\R^{+*}))$, the induced map at right will be $Id_{\Z}$. \\
The left $\tau_2$ map being obtained as the K-theory of the flip automorphism $f(x, y) \mapsto f(y, x)$, the left induced map is $-Id_{\Z}$ (see Appendix thm \ref{flipDemo}).

Then trough the isomorphisms in $(\ref{eq:isobetap})$, $\widetilde{\alpha_1}$ and $\widetilde{\alpha_2}$ will be identified with opposite automorphisms of $\Z$.
\end{proof}

Combining those two last lemmas provides us an identification of $\alpha_i^*$ as $(-1)^{i-1} Id_{K^{\diamond}(B)}$ with $\diamond = *$ or $\diamond = 1-*$ depending on the parity of $p$.

Then the connection $\delta^*_{\vert f}$ seen as a map $\delta^*_{\vert f} : K^*(B) \longrightarrow \underset{g' \in \mathcal{F}_{p-1}(X)}{\oplus} K^*(B)$ can be written as a combination of $-Id$, $Id$ and $0$ maps. To be more fluent, we will denote
\begin{equation}\label{eq:signeRel}
\sigma(f, g) = 
\left\{
    \begin{array}{ll}
        (-1)^{k-1} & \text{ if } g = g_{\widehat{k}} \\
        0 & \text{ otherwise}
    \end{array}
\right., 
\end{equation} \text{ as }\cite[eq (5.22)]{CRL18}.

\begin{prop}{}
	The connection maps of the sequence $0 \rightarrow \frac{A_{p-1}}{A_{p-2}} \rightarrow \frac{A_p}{A_{p-2}} \rightarrow \frac{A_p}{A_{p-1}} \rightarrow 0$ can be written as :
	\begin{tabular}{cccc}
		$\delta^* : $ & $\underset{f \in \mathcal{F}_{p}(X)}{\bigoplus} K^*(B)$ & $\longrightarrow$ & $\underset{g \in \mathcal{F}_{p-1}(X)}{\bigoplus} K^*(B)$ \\
		& $b_f$ & $\mapsto$ & $\underset{g \in \mathcal{F}_{p-1}(X)}{\Sigma} \sigma(f, g)b_f.$
	\end{tabular}
\end{prop}

\begin{rem}{\label{rq:connectionSurj}}
	When $X$ is connected, in the case $p=1$ we have : $\forall f \in \mathcal{F}_1(X), \sigma(f, \overset{\circ}{X}) = 1$. Then the last map becomes
	\begin{tabular}{ccc}
		$\underset{f \in \mathcal{F}_1(X)}{\bigoplus} K^*(B)$ & $\longrightarrow$ & $K^*(B)$ \\
		$b_f$ & $\mapsto$ & $b_f$.
	\end{tabular}
	In particular it tells us that the connection map $K^*(A_1/A_0) \rightarrow K^{1-*}(A_0)$ is surjective.
\end{rem}
\subsection{Pairing $K_*-H_*^{pcn}$}
Following the J.M. Lescure and P.Carrillo Rouse ideas, we would like to identify an homological behaviour allowing more computations starting from the connections previously computed \, $K_*(A_p/A_{p-1}) \longrightarrow K_{1-*}(A_{p-1}/A_{p-2})$. As in their case, the key belongs to conormal homology, but here we need to transport a K-theoretical data, not only integers. \\
In this section, we quickly define the conormal homology with coefficients, and because not every property J.M. Lescure and P.Carrillo Rouse use still hold in our case, we summerize those we will need. \\

For the whole section, we set $G$ to be an abelian group, and $X$ is a manifold with embedded corners of codimension $d$ (which can be the total space of a family, or not). 
\begin{defi}
	We set the following chain complex: \\
    $\xymatrix{\dots \ar[r] & C_{p+1}(X ; G) \ar[r]^{\partial_{p+1}} & C_{p}(X ; G) \ar[r]^{\partial_p} & C_{p-1}(X ; G) \ar[r] & \dots }$
    , \\
     with $C_p(X ; G) := C_p(X) \underset{\Z}{\otimes} G$, $C^p(X)$ being the $p-$chain of P. Carrillo Rouse and J.M Lescure \cite[eq (4.3)]{CRL18}. \\
     We define the complex differentials the same way tensoring with $G$ to get coefficients: 
     $\partial \Big( \gamma f \otimes \varepsilon \Big) := \underset{\underset{f \subset \overline{g}}{g \in \mathcal{F}_{p-1}(X)}}{\sum}	\gamma g \otimes (e_{i(g.f)} \lrcorner \varepsilon), \gamma \in G, $.
    
    The homology groups will be denoted $H_p^{cn}(X ; G)$. \\ And we define the periodised homology: $H_i^{pcn}(X ; G) := \bigoplus_{p \in \N} H^{cn}_{2p+i}(X ; G)$, $i=0,1$.
\end{defi}

As already mentioned in \cite[eq (4.9)]{CRL18}, there is a a natural filtration $\overset{\circ}{X} = X_0 \subset X_1 \subset \dots \subset X_d = X$, with $X_i = \{ x \in X : codim(x) \leqslant i \}$. Once again this allows us to define relative complexes:
\[
	\xymatrix{0 \ar[r] & C_{\bullet}(X_m, X_q ; G) \ar[r] & C_{\bullet}(X_l, X_q ; G) \ar[r] & C_{\bullet}(X_l, X_m ; G) \ar[r] & 0}, \, \, 
	-1 \leqslant q \leqslant m \leqslant l \leqslant d,
\]
using the convention $X_{-1} = \emptyset$.
This inducing an homology long exact sequence. Putting together the even terms (in $H_0^{pcn}$), and the odd terms (in $H_1^{pcn}$), we get a 6 term exact sequence:
\[\begin{tikzcd}
	{H_1^{pcn}(X_m, X_q ; G)} & {H_1^{pcn}(X_l, X_q ; G)} & {H_1^{pcn}(X_l, X_m ; G)} \\
	{H_0^{pcn}(X_l, X_m ; G)} & {H_0^{pcn}(X_l, X_q ; G)} & {H_0^{pcn}(X_m, X_q ; G)}
	\arrow[from=1-1, to=1-2]
	\arrow[from=1-2, to=1-3]
	\arrow["{\partial_1}", from=1-3, to=2-3]
	\arrow["{\partial_0}", from=2-1, to=1-1]
	\arrow[from=2-2, to=2-1]
	\arrow[from=2-3, to=2-2].
\end{tikzcd}\]

In suitable conditions, this sequence is actually easier to compute :
\begin{prop}
	If $X$ is connected and $\partial X \neq \emptyset$, the short exact sequence induced by, \\
	$\xymatrix{0 \ar[r] & C_{\bullet}(X_0 ; G) \ar[r] & C_{\bullet}(X ; G) \ar[r] & C_{\bullet}(X, X_0 ; G) \ar[r] & 0}$ is such that:
\[
	\xymatrix{0 \ar[r]^0 & H_1^{pcn}(X ; G) \ar@{^{(}->}[r] & H_1^{pcn}(X, X_0 ; G) \ar@{->>}[d]^{\partial_1} \\
        H_0^{pcn}(X, X_0 ; G) \ar@{->>}[u]^0 & H_0^{pcn}(X ; G) \ar[l]_{\cong} & H_0^{pcn}(X_0 ; G) \ar[l]_0}
\]
\end{prop}
\begin{proof}
	First we have $H_1^{pcn}(X_0 ; G) = 0$. \\
	Then having $H_0^{pcn}(X_0 ; G) = H_0^{cn}(X_0 ; G)$, a straightforward computation gives us the surjectivity of $\partial_1 : H_1^{pcn}(X, X_0 ; G) \rightarrow H_0^{cn}(X_0 ; G)$ (because $X$ is connected with non empty boundary). \\
	
	Moreover noticing that $\forall k \geqslant 1, H_k^{cn}(X_0 ; G) = 0$, the long exact sequence tells us that $\forall k \geqslant 2, H_k^{cn}(X ; G) \cong H_k^{cn}(X, X_0 ; G)$. For $k=0$, this also holds because of $\partial_1$ surjectivity and the vanishing of $H_1^{cn}(X_0 ; G)$. \\
	
	The nullity of the remaining maps are deduced.
\end{proof}

In particular , when $X$ is connected and $\partial X \neq \emptyset$ we have the short exact sequence:
\[
	\xymatrix{0 \ar[r] & H_1^{pcn}(X ; G) \ar[r] & H_1^{pcn}(X, X_0 ; G) \ar[r] & H_0^{pcn}(X_0 ; G) \ar[r] & 0}.
\]

Another conormal homology 6 term exact sequence will be helpful later:
\begin{thm}\label{diag:codim2Homology}
	If $X$ is an embedded corner manifold of codimension $d=2$, then the short exact sequence
\[
	\xymatrix{0 \ar[r] & C_{\bullet}(X_1, X_0 ; G) \ar[r] & C_{\bullet}(X, X_0 ; G) \ar[r] & C_{\bullet}(X, X_1 ; G) \ar[r] & 0},	
\]
induces a 6 term exact sequence:
\[\begin{tikzcd}
	{H_1^{pcn}(X_1, X_0 ; G)} & {H_1^{pcn}(X, X_0 ; G)} & 0 \\
	{H_0^{pcn}(X, X_1 ; G)} & {H_0^{pcn}(X, X_0 ; G)} & 0
	\arrow[from=1-1, to=1-2]
	\arrow[from=1-2, to=1-3]
	\arrow[from=1-3, to=2-3]
	\arrow[from=2-1, to=1-1]
	\arrow[from=2-2, to=2-1]
	\arrow[from=2-3, to=2-2]
\end{tikzcd}\]
\end{thm}

\begin{proof}
	The vanishing of $H_1^{pcn}(X, X_1 ; G)$ and $H_0^{pcn}(X_1, X_0 ; G)$ is a straightforward computation.
\end{proof}

Now we have more computational data with this homology, in particular in low codimension, we still have to relate it to K-theory.

\begin{rem}
	Both our computation (see (\ref{eq:signeRel})) and the conormal homology have a notion of \textit{relative sign} for two neighboring faces, because tuples are sorted they both coincide, we have: $\sigma(f, g) = (-1)^{i_{(g . f)}}$.
\end{rem}
Using this last fact, we can introduce the pairing :
\begin{thm}{$K^*-H^{pcn}$ pairing}{\label{thm:pairing}} \\
	For $k \in \{p-1, p\}$, the maps
	\begin{tabular}{ccc}
		$K^*(B)^{\# \mathcal{F}_k(X)}$ & $\longrightarrow$ & $C_k(X ; K^*(B))$ \\
		$b_f$ & $\mapsto$ & $b_f f \otimes e_{I_f}$,
	\end{tabular}
	where $\vert I \vert = p, f \subset H_{I_f}, * = 0, 1$, \\
	makes the diagram:
	\[\begin{tikzcd}
	{K^*(B)^{\# \mathcal{F}_p(X)}} & {K^*(B)^{\# \mathcal{F}_{p-1}(X)}} \\
	{C_p(X ; K^*(B))} & {C_{p-1}(X ; K^*(B))}
	\arrow["{\delta^*}", from=1-1, to=1-2]
	\arrow["\cong"{description}, no head, from=2-1, to=1-1]
	\arrow["\partial", from=2-1, to=2-2]
	\arrow["\cong"{description}, no head, from=2-2, to=1-2]
\end{tikzcd}\]
	commutative. \\
	In particular, $H_i^{pcn}(X_p, X_{p-1} ; G)$ being isomorphic to $C_p(X ; G)$ when $p$ and $i$ have the same parity, we have that:
	\begin{itemize}
        \item If $i \equiv p \mod 2$, 
        $\xymatrix{K_i(A_p / A_{p-1}) \ar[r] & K_{1-i}(A_{p-1}/A_{p-2}) \\
        H_i^{pcn}(X_p, X_{p-1} ; K^0(B)) \ar[u]^{\cong} \ar[r] & H_{1-i}^{pcn}(X_{p-1}, X_{p-2} ; K^0(B)) \ar[u]_{\cong}}$
        \item If $i \not \equiv p \mod 2$, 
        $\xymatrix{K_i(A_p / A_{p-1}) \ar[r] & K_{1-i}(A_{p-1}/A_{p-2}) \\
        H_{1-i}^{pcn}(X_p, X_{p-1} ; K^1(B)) \ar[u]^{\cong} \ar[r] & H_i^{pcn}(X_{p-1}, X_{p-2} ; K^1(B)) \ar[u]_{\cong}}$
    \end{itemize}
\end{thm}

\section{Obstruction computation for embedded families in low codimension}
Before proceeding to computations, we give here notations related to indexes we will use.

All our index maps will be realised using tangent groupoids. \\
If $X$ denotes a family of manifolds with embedded corners over a smooth base $B$, globally embedded, with $d := codim(X)$ we denote its associated index map: \\
\[
	Ind^X := \Big(K^0(\Gamma^{tan}_{b, \mathfrak{f}}(X)) \overset{e_1}{\longrightarrow} K^0(\Gamma_{b, \mathfrak{f}}(X))\Big).
\]
The associated $l$ times suspended index map is denoted:
\[
	Ind^X_l := \Big(K^0(\Gamma^{tan}_{b, \mathfrak{f}}(X) \times \R^{+*l}) \overset{e_1}{\longrightarrow} K^0(\Gamma_{b, \mathfrak{f}}(X) \times \R^{+*l})\Big).
\] \\
In particular when $X$ is smooth, $Ind^X$ is the classical family index computed in Atiyah Singer theorem for families (e.g see \cite{CR23}, p.14).

Using the filtration $\Gamma_{b, \mathfrak{f}}(X) = \underset{k=1}{\overset{d}{\cup}} \Gamma_{b, \mathfrak{f}}(X)_{\vert \mathcal{F}_{\leqslant k}}$ with $\Gamma_{b, \mathfrak{f}}(X)_{\vert \mathcal{F}_{\leqslant k}} := \underset{i=0}{\overset{k}{\bigsqcupsim}} \underset{f \in \mathcal{F}_i(X)}{\sqcup}(f \underset{B}{\times} f \times \R^{+*i})$, we provide another one on the tangent groupoid: $\Gamma^{tan}_{b, \mathfrak{f}}(X) = \Gamma^{tan}_{b, \mathfrak{f}}(X)_{\vert X \times [0, 1[} \sqcup \Gamma_{b, \mathfrak{f}}(X)$ that we denote: \\
\[
	\Gamma^{fred(k)}_{b, \mathfrak{f}}(X) := \Gamma^{tan}_{b, \mathfrak{f}}(X)_{\vert X \times [0, 1[} \sqcup \Gamma_{b, \mathfrak{f}}(X)_{\vert \mathcal{F}_{\leqslant k}}.
\] 
Elements of $C^*(\Gamma^{fred(k)}_{b, \mathfrak{f}}(X))$ will be called \textit{non commutative symbols}. In practice, $\Gamma^{fred(0)}_{b, \mathfrak{f}}(X)$ will be simply denoted $\Gamma^{fred}_{b, \mathfrak{f}}(X)$, and we notice that $\Gamma^{fred(d)}_{b, \mathfrak{f}}(X) = \Gamma^{tan}_{b, \mathfrak{f}}(X)$. \\
From $\Gamma^{fred(k)}_{b, \mathfrak{f}}(X)$, evaluation in $X \times \{1\}$ gives a non commutative index that we denote:
\[
	Ind^X_{nc(k)} := \Big( K^0(\Gamma^{fred(k)}_{b, \mathfrak{f}}(X)) \overset{e_1}{\rightarrow} K^0(\Gamma_{b, \mathfrak{f}}(X)_{\vert \mathcal{F}_{\leqslant k}}) = K_0(A_k) \Big)
\] 
Here again, when $k=0$ the index map will be simply denoted $Ind^X_{nc}$. \label{indiceNc}. As seen in remark \ref{grpdFred}, this map could carry Fredholm (resp. family) indices in the case of a manifold with boundary (resp. family with boundary). \\

Those non commutative indexes can be suspended as well, the $l^{th}$ suspended index map will be denoted:
\[
	Ind^X_{l, nc(k)} := \Big( K^0(\Gamma^{fred(k)}_{b, \mathfrak{f}}(X) \times \R^{+*l}) \overset{e_1}{\rightarrow} K^0(\Gamma_{b, \mathfrak{f}}(X)_{\vert \mathcal{F}_{\leqslant k}} \times \R^{+*l}) = K_0(A_k \otimes C_0(\R^{+*l})) \Big).
\] \\ \\

Following remark \ref{rk:symbVersTgt}, for the whole section, we will consider a symbol $\sigma_T \in  C^*(\Gamma^{tan}_{b, \mathfrak{f}}(X))$ associated to an elliptic operator $T$ such that $Ind_{\partial}([\sigma(T)]_1) = r_{\partial} \circ Ind^X([\sigma_T]_0)$.
\subsection{Boundary index for codimension 1 families globally embedded}

For that section $X \rightarrow B$ is a codimension 1 family of manifold with embedded corners ($X$ itself endowed with embedded corners), and $X$ is connected. \\
The canonical filtration reduced as $\emptyset = A_{-1} \subset A_0 \subset A_1 = C^*(\Gamma_{b, \mathfrak{f}}(X)$.

Using remark \ref{rq:connectionSurj}, the connections map in K-theory of $\xymatrix{0 \ar[r] & A_0 \ar[r] & A_1 \ar[r] & A_1 / A_0 \ar[r] & 0}$ are surjective. Then its 6 term exact sequence gives zero maps and injections providing the two following short exact sequences : \\
$\xymatrix{0 \ar[r] & K_*(A_1) \ar[r] & K_*(A_1 / A_0) \ar[r] & K_{1-*}(A_0) \ar[r] & 0}$, $*=0,1$. \\
In particular, with $*=0$, we see that $r_{\partial}$ is injective. \\
We compare it to conormal homology. \\ \\
We use the $K_*-H_*^{pcn}$ pairing (theorem \ref{thm:pairing}) with $H_1^{pcn}(X, X_0 ; K^{1-*}(B)) \overset{\partial}{\rightarrow} H_0^{pcn}(X_0 ; K^{1-*}(B))$ (putting this together with $Ker \, \partial = H_1^{cn}(X ; K^{1-*}(B)) = H_1^{pcn}(X ; K^{1-*}(B)$) to get a short exact sequence. \\ \\

We do get a conormal homology short exact sequence whose middle and right terms pair with the K-theory ones:
\[
\xymatrix{0 \ar[r] & K_*(A_1) \ar[r]^{r_{\partial}} & K_*(A_1 / A_0) \ar[r] & K_{1-*}(A_0) \ar[r] & 0 \\
0 \ar[r] & H_1^{pcn}(X ; K^{1-*}(B)) \ar[r] & H_1^{pcn}(X, X_0 ; K^{1-*}(B)) \ar[r] \ar[u]^{\cong} & H_0^{pcn}(X_0 ; K^{1-*}(B)) \ar[r] \ar[u]^{\cong} & 0 }
\]

In such a situation, we can use the 5's lemma to fill the diagram at left. The map will be automatically an isomorphism (actually it will be the restriction of the middle isomorphism). \\
At the end, everything is computable here:

\begin{thm}{}
    For $X$ a codimension 1 family of manifolds with embedded corners, we have:
    \begin{itemize}
        \item $K_*(A_0) \cong H_0^{pcn}(X_0 ; K^*(B)) \cong K^*(A_0)^{\# \mathcal{F}_0(X)}$
        \item $K_*(A_1/A_0) \cong H_1^{pcn}(X, X_0 ; K^{1-*}(B)) \cong K^{1-*}(B)^{\# \mathcal{F}_1(X)}$
        \item $K_*(A_1) \cong H_1^{pcn}(X ; K^{1-*}(B)) \cong K^{1-*}(B)^{\# \mathcal{F}_1(X) - 1}$. \\
    \end{itemize}
\end{thm}

\begin{proof}
	$K_*(A_1) \cong H_1^{pcn}(X ; K^{1-*}(B)) \cong H_1^{cn}(X ; K^{1-*}(B)) = Ker \partial_1$, \\
	where $\partial_1$ is computed as $(b_1, \dots, b_N) \in K^{1-*}(B)^{\# \mathcal{F}_1(X)} \mapsto b_1 + \dots + b_N \in K^{1-*}(B)$. \\
	Then $Ker \partial_1 \cong <(0, \dots, 0, b, -b, 0, \dots, 0)>_{b \in K^{1-*}(B)} \cong K^{1-*}(B)^{\# \mathcal{F}_1(X) - 1}$.
\end{proof}

The next step now is to caracterise the vanishing of $r_{\partial} \, \circ \, e_1$.
\[\begin{tikzcd}
	{K_0(A_1)} & {K_0(A_1/A_0)} \\
	{K^0(\Gamma^{tan}_{b, \mathfrak{f}}(X))}
	\arrow["{r_{\partial}}", hook, from=1-1, to=1-2]
	\arrow["{Ind^X}", from=2-1, to=1-1]
\end{tikzcd}\]

The maps $e_1$ and $r_{\partial}$ are only restrictions. The idea is now to express their composition differently. Rather than restricting the tangent groupoid to its slice $\{1\}$, then to its boundary, we will restrict the tangent groupoid to a sub tangent groupoid associated to the boundary, then we restrict it to the associated $\{1\}$ slice. The sub tangent groupoid will be identifiable as the tangent groupoid of a family of smooth manifolds, which will gives Atiyah Singer computable family indices. \\

Those restrictions interact this way:
\[\begin{tikzcd}
	& {K_0(A_1)} \\
	{K^0(\Gamma^{tan}_{b, \mathfrak{f}}(X))} && {K_0(A_1/A_0)} \\
	& {K^0(\Gamma^{tan}_{b, \mathfrak{f}}(g) \times \R^{+*})}
	\arrow["{r_{\partial}}", hook, from=1-2, to=2-3]
	\arrow["{Ind^X}", from=2-1, to=1-2]
	\arrow["{\underset{g \in \mathcal{F}_1(X)}{\oplus} r_g}"', from=2-1, to=3-2]
	\arrow["{\underset{g \in \mathcal{F}_1(X)}{\oplus} Ind^g_1}"', from=3-2, to=2-3]
\end{tikzcd}\]
We get the equality:  $Ind_{\partial}([\sigma_T]_0) = r_{\partial}(Ind^X([\sigma_T]_0)) = \left(Ind_1^g([{\sigma_T}_{\vert g}]_0)\right)_{g \in \mathcal{F}_1(X)}$. \\ Notice that  maps $Ind^g_1$ are computable as suspensions of family indices \\
Then we get the following result:
\premierCodThm

We can notice the following corollary :

\begin{thm}{Vanishing of the boundary index in codimension 1 : Case $K^1(B) = 0$} \\
	When $K^1(B) = 0$, the map $Ind_{\partial}$ is identically zero.
\end{thm}

\subsection{Boundary index for codimension 2 families globally embedded}

In this section, we caracterise the vanishing of the boundary index for a codimension 2 family globally endowed with embedded corners.
Now $\pi : \widetilde{X} \longrightarrow B$ denotes a codimension 2 family of manifold with embedded corners ($X$ itself endowed with embedded corners), $X$ is supposed to be connected. \\
Once again its Monthubert groupoid
$\Gamma_{b, \mathfrak{f}}(X) = 
\overset{\circ}{X} \underset{B}{\times} \overset{\circ}{X} \,
 \underset{\sim}{\bigsqcup} 
 \underset{g \in \mathcal{F}_1}{\sqcup} (g \underset{B}{\times} g \times \R^{*+}) \, 
 \underset{\sim}{\bigsqcup}
 \underset{g \in \mathcal{F}_2}{\sqcup} (g \underset{B}{\times} g \times \R^{*+2})$ will be endowed with the obvious filtration of his full, supposed amenable, $C^*$- algebra: $\emptyset = A_{-1} \subset A_0 \subset A_1 \subset A_2 = C^*(\Gamma_{b, \mathfrak{f}}(X))$. \\ 

First we recall that from the short exact sequence $0 \rightarrow \frac{A_1}{A_0} \rightarrow \frac{A_2}{A_0} \rightarrow \frac{A_2}{A_1} \rightarrow 0$ we can produce  a 6 terms short exact sequence
\[\begin{tikzcd}
	{K_1(A_1/A_0)} & {K_1(A_2/A_0)} & {K_1(A_2/A_1)} \\
	{K_0(A_2/A_1)} & {K_0(A_2/A_0)} & {K_0(A_1/A_0)}
	\arrow[from=1-1, to=1-2]
	\arrow[from=1-2, to=1-3]
	\arrow["{d_{2,1}}", from=1-3, to=2-3]
	\arrow["{d_{2,0}}", from=2-1, to=1-1]
	\arrow[from=2-2, to=2-1]
	\arrow[from=2-3, to=2-2]
\end{tikzcd}\]
where the morphisms $d_{2,*}$, $*=0, 1$, are explicitely computed in terms of conormal homology. \\
From this we can force the short exactness to get:

\begin{equation}{\label{diag:encadrementCodim2}}
	   0 \longrightarrow
	\frac{K_0(A_1/A_0)}{Im \, d_{2, 1}} \longrightarrow
	K_0(A_2/A_0) \longrightarrow
	Ker \, d_{2,0} \longrightarrow 0.
\end{equation}

Those left and right terms can be computed using 6 term exact sequence in conormal homology (theorem \ref{diag:codim2Homology}) and the $K_*-H_*^{pcn}$ pairing (theorem \ref{thm:pairing}) to extend the correspondance as follows: \\

{\small
\[\begin{tikzcd}
	0 & {Ker \, d_{2,*}} & {K_*(A_2/A_1)} & {K_{1-*}(A_1/A_0)} & {\frac{K_{1-*}(A_1/A_0)}{Im \, d_{2,*}}} & 0 \\
	0 & {H_0^{pcn}(X, X_0 ; K^*(B))} & {H_0^{pcn}(X, X_1 ; K^*(B))} & {H_1^{pcn}(X_1, X_0 ; K^*(B))} & {H_1^{pcn}(X, X_0 ; K^*(B))} & 0
	\arrow[from=1-1, to=1-2]
	\arrow[from=1-2, to=1-3]
	\arrow[from=1-3, to=1-4]
	\arrow[from=1-4, to=1-5]
	\arrow[from=1-5, to=1-6]
	\arrow[from=2-1, to=2-2]
	\arrow[dotted, from=2-2, to=1-2]
	\arrow[from=2-2, to=2-3]
	\arrow["\cong"', from=2-3, to=1-3]
	\arrow[from=2-3, to=2-4]
	\arrow["\cong", from=2-4, to=1-4]
	\arrow[from=2-4, to=2-5]
	\arrow[dotted, from=2-5, to=1-5]
	\arrow[from=2-5, to=2-6]
\end{tikzcd}\]} \\

Using the five's lemma gives us two isomorphims which are only restriction and quotient of the ones we already know. Applying them to the sequence (\ref{diag:encadrementCodim2}) gives us our first obstruction result:

\secondObstrctSpaceThm

To express the obstruction we are looking for, we will caracterise the vanishing of the boundary index $r_{\partial} \circ Ind^X : K^0(\Gamma^{tan}_{b, \mathfrak{f}}(X)) \longrightarrow K_0(A_2/A_0)$ in terms of smaller indices and conormal homology. \\

To caracterise the vanishing of $Ind_{\partial}$ we begin to notice that the restriction morphism $C^*(A_2/A_0) \longrightarrow C^*(A_2/A_1)$ gives the commutative triangle:
\begin{equation}\label{diag:triangler2}
	\begin{tikzcd}
		{K_0(A_2)} && {K_0(A_2/A_0)} \\
		&& {K_0(A_2/A_1)}
		\arrow["{r_{\partial}}", hook, from=1-1, to=1-3]
		\arrow["{r_2}"', from=1-1, to=2-3]
		\arrow[from=1-3, to=2-3]
	\end{tikzcd}
\end{equation}

\begin{rem}{}
	The map $r_{\partial} : K_0(A_2) \hookrightarrow K_0(A_2/A_0)$ is injective because in the short exact sequence morphism:
\[\begin{tikzcd}
	0 & {A_0} & {A_2} & {A_2/A_0} & 0 \\
	0 & {A_0} & {A_1} & {A_1/A_0} & 0,
	\arrow[from=1-1, to=1-2]
	\arrow[from=1-2, to=1-3]
	\arrow[from=1-3, to=1-4]
	\arrow[from=1-4, to=1-5]
	\arrow[from=2-1, to=2-2]
	\arrow["{=}"{description}, no head, from=2-2, to=1-2]
	\arrow[from=2-2, to=2-3]
	\arrow[from=2-3, to=1-3]
	\arrow[from=2-3, to=2-4]
	\arrow[from=2-4, to=1-4]
	\arrow[from=2-4, to=2-5]
\end{tikzcd}\]
we know from remark \ref{rq:connectionSurj} that the lower row has surjective connection maps. Then the upper row too. The 6 term exact sequence associated to the upper row gives then the injectivity.
\end{rem}

The codimension of $X$ being equal to $2$, the codimension $2$ faces have disjoint closures, which will allows us to compute $r_2$ facewise, moreover we have explicitely $K_0(A_2/A_1) \cong K^0(B)^{\mathcal{F}_2(X)}$. Using this we have the following lemma: \\

\begin{lem}
		$r_2(Ind^X([\sigma_T]_0)) = 
		\left( Ind_2^f([{\sigma_T}_{\vert f}]_0) \right)_{f \in \mathcal{F}_2(X)}$.
\end{lem} 

\begin{proof}
	The expression comes from the commutativity of the following diagram:
\[\begin{tikzcd}
	& {K_0(A_2)} \\
	{K^0(\Gamma_{b, \mathfrak{f}}^{tan}(X))} && {K_0(A_2/A_1)} \\
	& {\underset{f \in \mathcal{F}_2(X)}{\bigoplus} K^0(\Gamma_{b, \mathfrak{f}}^{tan}(X) \times \R^{+*2})}
	\arrow["{\small{r_2}}", from=1-2, to=2-3]
	\arrow["{\small{Ind^X}}"{pos=0.4}, from=2-1, to=1-2]
	\arrow["{\tiny{\underset{f \in \mathcal{F}_2(X)}{\oplus} r_f}}"'{pos=0.3}, from=2-1, to=3-2]
	\arrow["{\tiny{\underset{f \in \mathcal{F}_2(X)}{\bigoplus} Ind^f_2}}"'{pos=0.7}, from=3-2, to=2-3]
\end{tikzcd}\]
\end{proof}
With $r_2$ we have a computable index. The diagram (\ref{diag:triangler2}) tells us that $Ind_{\partial}([\sigma_T]_0)=  r_{\partial}(Ind^X([\sigma_T]_0)) = 0$ implies $r_2(Ind^X([\sigma_T]_0)) = 0$. To examine the reciprocal let's see what the vanishing of $r_2(Ind^X([\sigma_T]_0))$ brings us. \\ \\
The open close decomposition, see (\ref{ouvertferme}) in section 1, $\Gamma_{b, \mathfrak{f}}^{tan}(X) = \Gamma^{fred(1)}_{b, \mathfrak{f}}(X) \sqcup \Gamma_{b, \mathfrak{f}}(X)_{\vert \mathcal{F}_2(X)}$ being compatible with the decomposition $\Gamma_{b, \mathfrak{f}}(X) = \Gamma_{b, \mathfrak{f}}(X)_{\vert \mathcal{F}_{\leqslant 1}(X)} \sqcup \Gamma_{b, \mathfrak{f}}(X)_{\vert \mathcal{F}_2(X)}$ exaluation on slice $X \times \{1\}$ provides a morphism of exact sequences: \\
\[\begin{tikzcd}
	{K_0(A_1)} & {K_0(A_2)} & {K_0(A_2/A_1)} \\
	{K^0(\Gamma_{b, \mathfrak{f}}^{fred(1)}(X)} & {K^0(\Gamma^{tan}_{b, \mathfrak{f}}(X))} & {K_0(A_2/A_1)}
	\arrow[from=1-1, to=1-2]
	\arrow["{r_2}", from=1-2, to=1-3]
	\arrow["{Ind_{nc(1)}^{X}}", from=2-1, to=1-1]
	\arrow[from=2-1, to=2-2]
	\arrow["{Ind^X}", from=2-2, to=1-2]
	\arrow[from=2-2, to=2-3]
	\arrow["{=}"{description}, no head, from=2-3, to=1-3]
\end{tikzcd}.\]
From lines exactenss we can deduce the following property:
\begin{prop}
	Denoting $\underline{.}$ the extension by zero map $C^*(\Gamma_{b, \mathfrak{f}}^{fred(1)}(X)) \longrightarrow C^*(\Gamma_{b, \mathfrak{f}}^{tan}(X))$: \\	
	\[
		r_2(Ind^X([\sigma_T]_0)) = 0
		\Leftrightarrow
		\exists \phi \in C^*(\Gamma_{b, \mathfrak{f}}^{fred(1)}(X)) \text{ such that } \underline{\phi} = \sigma_T.
	\]
\end{prop}
In what follows, such a lift $\phi$ will be denoted $\overset{\circ}{\sigma_T}$. In particular $\underline{\overset{\circ}{\sigma_T}} = \sigma_T$. \\ \\
The $r_2$ vanishing gives us a lift to a non commutative symbol $\overset{\circ}{\sigma_T}$, the next step is then to understand the vanishing of $Ind^X([\underline{\overset{\circ}{\sigma_T}}]_0)$. This is the aim of the following lemma:
\begin{lem}
	Let $\phi \in C^*(\Gamma_{b, \mathfrak{f}}^{fred(1)}(X))$ a non commutative symbol. \\
	Then:
	\[
		r_{\partial} \circ Ind^X([\underline{\phi}]_0) = 0
		\Leftrightarrow
		\left[r_1 \circ Ind_{nc(1)}^X([\phi]_0)
		\right]_{H_1^{pcn}(X, X_0 ; K^1(B))} = 0 \, \, \, \textit{(a cycle for that homology).}
	\]
\end{lem}

\begin{proof}
	From the commutative diagram:
\[\begin{tikzcd}
	{K_1(A_2/A_1)} & {K_0(A_1/A_0)} & {K_0(A_2/A_0)} \\
	& {K_0(A_1)} & {K_0(A_2)} \\
	& {K^0(\Gamma_{b, \mathfrak{f}}^{fred(1)}(X))} & {K^0(\Gamma^{tan}_{b, \mathfrak{f}}(X))}
	\arrow["{d_{2, 1}}"', from=1-1, to=1-2]
	\arrow[from=1-2, to=1-3]
	\arrow["{r_1}", hook, from=2-2, to=1-2]
	\arrow[from=2-2, to=2-3]
	\arrow["{r_{\partial}}", hook, from=2-3, to=1-3]
	\arrow["{Ind_{nc(1)}^X}", from=3-2, to=2-2]
	\arrow[from=3-2, to=3-3]
	\arrow["{Ind^X}", from=3-3, to=2-3]
\end{tikzcd}\]
(the first line is exact) we have the equivalence: \\
\[
	r_{\partial} \circ Ind^X([\underline{\phi}]_0) = 0
	\Leftrightarrow
	r_1 \circ Ind_{nc(1)}^X([\phi]_0) \in Im \, d_{2,1}
\]
Now to caracterise $Im \, d_{2,1}$ we consider the following exact sequence morphism:
\[\begin{tikzcd}
	{C_2(X ; K^1(B))} & {C_1(X ; K^1(B))} & {H_1^{pcn}(X, X_0 ; K^1(B))} \\
	{K_1(A_2/A_1)} & {K_0(A_1/A_0)} & {K_0(A_2/A_0)}
	\arrow["{\partial_2}", from=1-1, to=1-2]
	\arrow["\cong"{description}, no head, from=1-1, to=2-1]
	\arrow["{[.]_{H_1^{pcn}}}", from=1-2, to=1-3]
	\arrow["\cong"{description}, no head, from=1-2, to=2-2]
	\arrow[hook', from=1-3, to=2-3]
	\arrow["{d_{2, 1}}"', from=2-1, to=2-2]
	\arrow[from=2-2, to=2-3]
\end{tikzcd}\]
the isomorphisms at the left and the middle are obtained from the $K-H^{pcn}$ pairing (theorem \ref{thm:pairing}), and the injection is obtained using the isomorphism theorem on the morphism $K_0(A_1/A_0) \longrightarrow K_0(A_2/A_0)$. \\
From this diagram we have that $Im \, d_{2,1} \cong Ker \, [.]_{H_1^{pcn}}$ via the (canonical) isomorphism of the pairing.
\end{proof}
Then the conormal homology of $r_1 \circ Ind_{nc(1)}^X([\phi]_0)$ provides us a combinatoric condition for the vanishing of the boundary index of a lifted symbol $\phi$. To make this last condition more explicite, we use codimension 1 results. \\

The key fact is to notice that for every $g \in \mathcal{F}_1(X)$, the closure $\overline{g}$ is a family of codimension $1$ manifolds with embedded corners. And also that $\Gamma_{b, \mathfrak{f}}^{fred(1)}(X)$ contains suspended Fredholm groupoids of those $\overline{g}$. Following the facewise compuation in this context gives us:
\begin{prop}
	Let $\phi \in C^*(\Gamma^{fred(1)}_{b, \mathfrak{f}})$ a non commutative symbol. \\
	Then $r_1 \circ Ind_{nc(1)}^X([\phi]_0) = \left( Ind_{1, nc(1)}^{\overline{g}}([\phi_{\vert \overline{g}}]_0) \right)_{g \in \mathcal{F}_1(X)}$.
\end{prop}

\begin{proof}
	Ths compuation comes from the commutativity of the following:
\[\begin{tikzcd}
	& {K_0(A_1)} \\
	{K^0(\Gamma_{b,\mathfrak{f}}^{fred(1)}(X)} && {K_0(A_1/A_0)} \\
	& {\underset{g \in \mathcal{F}_1(X)}{\bigoplus} K^0(\Gamma^{fred}_{b, \mathfrak{f}}(\overline{g}) \times \R^{+*})}
	\arrow["{r_1}", hook, from=1-2, to=2-3]
	\arrow["{Ind_{nc(1)}^X}", from=2-1, to=1-2]
	\arrow["{\underset{g \in \mathcal{F}_1(X)}{\oplus} r_{\overline{g} \times[0, 1] \setminus \partial \overline{g} \times \{1\}}}"'{pos=0.2}, from=2-1, to=3-2]
	\arrow["{\underset{g \in \mathcal{F}_1(X)}{\oplus} Ind_{1, nc}^{\overline{g}}}"'{pos=1}, from=3-2, to=2-3]
\end{tikzcd}\]
\end{proof}

We can now compile those ingredients in what will be our main result:
\secondCodThm

We can notice a different behaviour depending on the codimension. For the codimension $2$ faces, which have disjoint closure, we only ask for the vanishing of their suspended associated index. The codimension $1$ faces have closure which can intersect. This can allow cancelling phenomenas between them. \\

The combinatoric properties are expressed using homology groups with coefficients in $K^1(B)$. There are a lot of cases where this group is zero (e.g the family with one membre: $B=\{*\}$, in this case we recover the results of \cite{CRL18} and \cite{CRL25}).
\begin{thm}{Caracterisation of $ind_{\partial}$ vanishing - Case $K^1(B) = 0$}
	\[
			K_0(A_2/A_0) \cong H_0^{pcn}(X, X_0;  K^0(B))
			\,\,\, \text{and} \,\,\,
			Ind_{\partial}([\sigma(T)]_0) = 0
			\Leftrightarrow
			\forall f \in \mathcal{F}_2(X), \, Ind^f_2([{\sigma_T}_{\vert f}]_0) = 0
	\]
\end{thm}

\section{Appendix}

\begin{prop}\label{flipDemo}
	Let $F : C_0(\R^2) \rightarrow C_0(\R^2)$ the flip automorphism defined as $F(f)(x, y) = f(y, x)$. Then under the isomorphism $K_0(C_0(\R^2)) \cong \Z$, $K_0(F)$ is identified with $-Id_{\Z}$.
\end{prop}

\begin{proof}
	The flip automorphism $F$ being involutive, then $K_0(F)$ is involutive as well. Through the isomoprhism $K_0(C_0(\R^2)) \cong \Z$ it will give an involution of $\Z$, i.e $\pm Id_{\Z}$. \\ Then $K_0(F) = \pm Id_{K_0(C_0(\R^2))}$. \\ \\
	Let $\mathbb{D} \subset \mathbb{C}$ the closed unit disk. The isomorphism $\R^2 \cong \overset{\circ}{\mathbb{D}}$ induce a $C^*$- algebra isomorphism through which the morphism $F$ is written as
	\begin{tabular}{ccc}
		$C_0(\overset{\circ}{\mathbb{D}})$ & $\rightarrow$ & $C_0(\overset{\circ}{\mathbb{D}})$ \\
		$f$ & $\mapsto$ & $i \overline{f}$.
	\end{tabular}
	Using this morphism on the whole unit disk $\mathbb{D}$ rise a short exact sequence morphism:
\[\begin{tikzcd}
	0 & {C_0(\R^2)} & {C_0(\mathbb{D})} & {C(S^1)} & 0 \\
	0 & {C_0(\R^2)} & {C_0(\mathbb{D})} & {C(S^1)} & 0
	\arrow[from=1-1, to=1-2]
	\arrow[from=1-2, to=1-3]
	\arrow["F"', from=1-2, to=2-2]
	\arrow[from=1-3, to=1-4]
	\arrow[from=1-3, to=2-3]
	\arrow[from=1-4, to=1-5]
	\arrow["G", from=1-4, to=2-4]
	\arrow[from=2-1, to=2-2]
	\arrow[from=2-2, to=2-3]
	\arrow[from=2-3, to=2-4]
	\arrow[from=2-4, to=2-5]
\end{tikzcd}\]
\end{proof}
with $G(g) = i \overline{g}$. \\
Using connection maps naturality, it is enough to know that $K_1(G) : K_1(C(S^1)) \rightarrow K_1(C(S^1))$ corresponds to $-Id_{K_1(C(S^1))}$ to conclude. \\
Using $K_1(C(S^1)) \cong \Z$, and $G$ being involutive as well, once again $K_1(G) = \pm Id_{K_1(C(S^1))}$. Let $g \in C(S^1)$ the function defined by $g(z) = z$ whose class do not vanish in $K_1(C(S^1))$. \\
We have: $[g]_1 + [\overline{g}]_1 = [g \oplus \overline{g}]_1 \underset{\tiny{Wh. \, lem.}}{=} [g \overline{g} \oplus 1]_1 = [1 \oplus 1]_1 = 0_{K_1(C(S^1))}$, then $[\overline{g}]_1 = -[g]_1$. The constant function equal to $i$ being homotopic to the constant $1$ function we also have $[i]_1 = [1]_1 = 0$. \\ \\
Finally $K_1(G)([g]_1) = [i\overline{g}]_1 = [i]_1 + [\overline{g}]_1 = -[g]_1$. Then $K_1(G)$ could not be equal to identity, which conclude.

\printbibliography

@article{CRL18,
 author = {Carillo Rouse, Paulo and Lescure, Jean-Marie},
 title = {Geometric obstructions for {Fredholm} boundary conditions for manifolds with corners},
 fjournal = {Annals of \(K\)-Theory},
 journal = {Ann. \(K\)-Theory},
 issn = {2379-1683},
 volume = {3},
 number = {3},
 pages = {523--563},
 year = {2018},
 language = {English},
 doi = {10.2140/akt.2018.3.523},
 zbMATH = {6911676},
 Zbl = {1394.19003}
}

@misc{CRL25,
 author = {Carrillo Rouse, Paulo and Lescure, Jean-Marie},
 title = {Fredholm anomalies on manifold with corners of low codimensions and conormal corner cycles},
 year = {2025},
 howpublished = {Preprint, {arXiv}:2501.05071 [math.{KT}] (2025)},
 url = {https://arxiv.org/abs/2501.05071},
 arXiv = {arXiv:2501.05071}
}

@book{HR00,
 author = {Higson, Nigel and Roe, John},
 title = {Analytic {K}-homology},
 fseries = {Oxford Mathematical Monographs},
 series = {Oxford Math. Monogr.},
 isbn = {0-19-851176-0},
 year = {2000},
 publisher = {Oxford: Oxford University Press},
 language = {English},
 zbMATH = {1557174},
 Zbl = {0968.46058}
}

@article{NWX99,
 author = {Nistor, Victor and Weinstein, Alan and Xu, Ping},
 title = {Pseudodifferential operators on differential groupoids},
 fjournal = {Pacific Journal of Mathematics},
 journal = {Pac. J. Math.},
 issn = {1945-5844},
 volume = {189},
 number = {1},
 pages = {117--152},
 year = {1999},
 language = {English},
 doi = {10.2140/pjm.1999.189.117},
 zbMATH = {1344319},
 Zbl = {0940.58014}
}

@thesis{CR23,
  title={Index theory via deformation groupoids},
  author={Carrillo Rouse, Paulo},
  year={2023},
  school={Universit{\'e} Paul Sabatier (Toulouse 3)}
}

@article{Wil19,
  title={A Tool Kit for Groupoid C*-Algebras},
  author={Dana P. Williams},
  journal={Mathematical Surveys and
                        Monographs},
  year={2019},
  url={https://api.semanticscholar.org/CorpusID:216570405}
}

@unpublished{MonthubertHDR2005,
  author       = {Bertrand Monthubert},
  title        = {Contribution de la géométrie non commutative à la théorie de l'indice sur les variétés singulières},
  note         = {Habilitation à diriger des recherches, Université Paul Sabatier, soutenue le 2 décembre 2005},
  year         = {2005},
  url          = {https://www.math.univ-toulouse.fr/~monthube/articles/habilitation.pdf},
}

@misc{Obs21,
 author = {Obster, Lennart},
 title = {Blow-ups of {Lie} groupoids and {Lie} algebroids},
 year = {2021},
 howpublished = {Preprint, {arXiv}:2110.12247 [math.{DG}] (2021)},
 url = {https://arxiv.org/abs/2110.12247},
 arXiv = {arXiv:2110.12247}
}

@misc{PW19,
      title={Equivariant control data and neighborhood deformation retractions}, 
      author={Markus J. Pflaum and Graeme Wilkin},
      year={2019},
      eprint={1706.09539},
      archivePrefix={arXiv},
      primaryClass={math.SG},
      url={https://arxiv.org/abs/1706.09539}, 
}

@article{Tro78,
  author={Trotman, David J. A.},
  title={topologiques des conditions de Whitney, Interpr{\'e}tations},
  journal={Ast{\'e}risque},
  volume={59},
  number={60},
  pages={233--248},
  year={1978}
}

@article{Tro79,
author = {Trotman, David J. A.},
journal = {Annales scientifiques de l'École Normale Supérieure},
language = {eng},
number = {4},
pages = {453-463},
publisher = {Elsevier},
title = {Geometric versions of Whitney regularity for smooth stratifications},
url = {http://eudml.org/doc/82040},
volume = {12},
year = {1979},
}

@misc{SV22,
      title={Conormal homology of manifolds with corners}, 
      author={Thomas Schick and Mario Velasquez},
      year={2022},
      eprint={2111.13601},
      archivePrefix={arXiv},
      primaryClass={math.KT},
      url={https://arxiv.org/abs/2111.13601}, 
}

@misc{RLM13,
      title={A cohomological formula for the Atiyah-Patodi-Singer index on manifolds with boundary}, 
      author={Carrillo Rouse, Paulo and Lescure, Jean-Marie and Monthubert, Bertrand},
      year={2013},
      eprint={1207.3514},
      archivePrefix={arXiv},
      primaryClass={math.OA},
      url={https://arxiv.org/abs/1207.3514}, 
}

@inproceedings{AB65,
  title={The index theorem for manifolds with boundary},
  author={Atiyah, MF and Bott, R},
  booktitle={Appendix I in seminar on the Atiyah-Singer index theorem, Annals. of Math. Stud},
  volume={57},
  pages={337--351},
  year={1965}
}

@article{BC90a,
  title={Families index for manifolds with boundary, superconnections, and cones. I. Families of manifolds with boundary and Dirac operators},
  author={Bismut, Jean-Michel and Cheeger, Jeff},
  journal={Journal of functional analysis},
  volume={89},
  number={2},
  pages={313--363},
  year={1990},
  publisher={Elsevier}
}

@article{BC90b,
  title={Families index for manifolds with boundary, superconnections and cones. II. The Chern character},
  author={Bismut, Jean-Michel and Cheeger, Jeff},
  journal={Journal of functional analysis},
  volume={90},
  number={2},
  pages={306--354},
  year={1990},
  publisher={Elsevier}
}

@article{MP97,
  title={An index theorem for families of Dirac operators on odd-dimensional manifolds with boundary},
  author={Melrose, Richard B and Piazza, Paolo},
  journal={Journal of Differential Geometry},
  volume={46},
  number={2},
  pages={287--334},
  year={1997},
  publisher={Lehigh University}
}

@article{MN11,
   title={A topological index theorem for manifolds with corners},
   volume={148},
   ISSN={1570-5846},
   url={http://dx.doi.org/10.1112/S0010437X11005458},
   DOI={10.1112/s0010437x11005458},
   number={2},
   journal={Compositio Mathematica},
   publisher={Wiley},
   author={Monthubert, Bertrand and Nistor, Victor},
   year={2011},
   month=nov, pages={640–668}
}

@incollection{DS19,
  title={Lie groupoids, pseudodifferential calculus, and index theory},
  author={Debord, Claire and Skandalis, Georges},
  booktitle={Advances in Noncommutative Geometry: On the Occasion of Alain Connes' 70th Birthday},
  pages={245--289},
  year={2020},
  publisher={Springer}
}

@article{NSS10,
  title={Atiyah--Bott index on stratified manifolds},
  author={Nazaikinskii, Vladimir Evgen'evich and Savin, A Yu and Sternin, B Yu},
  journal={Journal of Mathematical Sciences},
  volume={170},
  number={2},
  pages={229--237},
  year={2010},
  publisher={Springer}
}

@incollection{NSS08,
  title={Elliptic theory on manifolds with corners: II. homotopy classification and K-homology},
  author={Nazaikinskii, Vladimir and Savin, Anton and Sternin, Boris},
  booktitle={C*-algebras and Elliptic Theory II},
  pages={207--226},
  year={2008},
  publisher={Springer}
}

@article{CRLV21,
   title={On Fredholm boundary conditions on manifolds with corners, I: Global corner cycles obstructions},
   volume={6},
   ISSN={2379-1683},
   url={http://dx.doi.org/10.2140/akt.2021.6.607},
   DOI={10.2140/akt.2021.6.607},
   number={4},
   journal={Annals of K-Theory},
   publisher={Mathematical Sciences Publishers},
   author={Carrillo Rouse, Paulo and Lescure, Jean-Marie and Velásquez, Mario},
   year={2021},
   month=dec, pages={607–628} }

\end{document}